\documentclass[12pt]{article}
\usepackage{latexsym}
\usepackage{amsmath}
\usepackage{amssymb}
\usepackage{amscd}
\usepackage{array}
\usepackage{comment}
\usepackage{amsthm}
\usepackage{amsfonts}
\usepackage{enumerate}
\usepackage{hyperref}
\usepackage{stackrel}
\newtheorem{theorem}{Theorem}[]
\newtheorem{proposition}[theorem]{Proposition}
\newtheorem{lemma}[theorem]{Lemma}
\newtheorem{corollary}[theorem]{Corollary}

\theoremstyle{definition}
\newtheorem{definition}[theorem]{Definition}

\newcommand{\Z}{\mathbf Z}

\newcommand{\Gal}{\mathrm{Gal}}

\newcommand{\mo}{\mathrm{mod}\, }
\newcommand{\Stab}{\mathrm{Stab}}
\newcommand{\Hol}{\mathrm{Hol}}

\newcommand{\Sym}{\operatorname{Sym}}

\newcommand{\GL}{\mathrm{GL}}

\newcommand{\End}{\operatorname{End}}

\newcommand{\Aut}{\operatorname{Aut}}

\newcommand{\Id}{\operatorname{Id}}
\newcommand{\Syl}{\operatorname{Syl}}
\newcommand{\wL} {{\widetilde{L}}}
\newcommand{\Soc}{\operatorname{Soc}}
\newcommand{\Ann}{\operatorname{Ann}}

\begin{document}

\large
\begin{center}
{\bf Hopf Galois structures on extensions of degree twice \\ an odd prime square and their associated skew left braces}

\vspace{0.4cm}
Teresa Crespo

\vspace{0.2cm}

\normalsize{Departament de Matem\`atiques i Inform\`atica, Universitat de Barcelona, \\ Gran Via de les Corts Catalanes 585, 08007 Barcelona, Spain, \\
e-mail: teresa.crespo@ub.edu}
\end{center}

\normalsize

\begin{abstract} We determine the Hopf Galois structures on a Galois field extension of degree twice an odd prime square and classify the corresponding skew left braces. Besides we determine the separable field extensions of degree twice an odd prime square allowing a cyclic Hopf Galois structure and the number of these structures.
\end{abstract}

\section{Introduction}
A Hopf Galois structure on a finite extension of fields $L/K$ is a pair $(H,\mu)$, where $H$ is
a finite cocommutative $K$-Hopf algebra  and $\mu$ is a
Hopf action of $H$ on $L$, i.e. a $K$-linear map $\mu: H \to
\End_K(L)$ giving $L$ a left $H$-module algebra structure and inducing a $K$-vector space isomorphism $L\otimes_K H\to\End_K(L)$.
Hopf Galois structures were introduced by Chase and Sweedler in \cite{C-S}.
For separable field extensions, Greither and
Pareigis \cite{G-P} give the following group-theoretic
equivalent condition to the existence of a Hopf Galois structure.

\begin{theorem}\label{G-P}
Let $L/K$ be a separable field extension of degree $g$, $\wL$ its Galois closure, $G=\Gal(\wL/K), G'=\Gal(\wL/L)$. Then there is a bijective correspondence
between the set of isomorphism classes of Hopf Galois structures on $L/K$ and the set of
regular subgroups $N$ of the symmetric group $S_g$ normalized by $\lambda_G(G)$, where
$\lambda_G:G \hookrightarrow S_g$ is the monomorphism given by the action of
$G$ on the left cosets $G/G'$.
\end{theorem}

For a given Hopf Galois structure on a separable field extension $L/K$ of degree $g$, we will refer to the isomorphism class of the corresponding group $N$ as the type of the Hopf Galois
structure. The Hopf algebra $H$ corresponding to a regular subgroup $N$ of $S_g$ normalized by $\lambda_G (G)$ is the $K$-Hopf subalgebra $\wL[N]^G$ of the group algebra $\wL[N]$ fixed under the action of $G$, where $G$ acts on $\wL$ by $K$-automorphisms and on $N$ by conjugation through $\lambda_G$. The Hopf action is induced by $n \mapsto n^{-1}(\overline{1})$, for $n \in N$, where we identify $S_g$ with the group of permutations of $G/G'$ and $\overline{1}$ denotes the class of $1_G$ in $G/G'$.

Childs \cite{Ch1} gives an equivalent  condition to the existence of a Hopf Galois structure introducing the holomorph of the regular subgroup $N$ of $S_g$. Let $\lambda_N:N\to \Sym(N)$ be the morphism given by the action of
$N$ on itself by left translation. The holomorph $\Hol(N)$ of $N$ is the normalizer of $\lambda_N(N)$ in $\Sym(N)$. As abstract groups, we have $\Hol(N)=N\rtimes \Aut(N)$. We state the more precise formulation of Childs' result due to Byott \cite{B} (see also \cite{Ch2} Theorem 7.3).

\begin{theorem}\label{theoB} Let $G$ be a finite group, $G'\subset G$ a subgroup and $\lambda_G:G\to \Sym(G/G')$ the morphism given by the action of
$G$ on the left cosets $G/G'$.
Let $N$ be a group of
order $[G:G']$ with identity element $e_N$. Then there is a
bijection between
$$
{\cal N}=\{\alpha:N\hookrightarrow \Sym(G/G') \mbox{ such that
}\alpha (N)\mbox{ is regular}\}
$$
and
$$
{\cal G}=\{\beta:G\hookrightarrow \Sym(N) \mbox{ such that }\beta
(G')\mbox{ is the stabilizer of } e_N\}
$$
Under this bijection, if $\alpha, \alpha' \in {\cal N}$ correspond to
$\beta, \beta' \in {\cal G}$, respectively, then $\alpha(N)=\alpha'(N)$ if and only if $\beta(G)$ and $\beta'(G)$ are conjugate by an element of $\Aut(N)$; and  $\alpha(N)$ is normalized by
$\lambda_G(G)$ if and only if $\beta(G)$ is contained in the
holomorph $\Hol(N)$ of $N$.
\end{theorem}

As a corollary to the preceding theorem Byott \cite{B}, Proposition 1, obtains the following formula to count Hopf Galois structures.

\begin{corollary}\label{cor} Let $L/K$ be a separable field extension of degree $g$, $\wL$ its Galois closure, $G=\Gal(\wL/K), G'=\Gal(\wL/L)$. Let $N$ be an abstract group of order $g$ and let $\Hol(N)$ denote the holomorph of $N$. The number $a(N,L/K)$ of Hopf Galois structures of type $N$ on $L/K$ is given by the following formula

$$a(N,L/K)= \dfrac{|\Aut(G,G')|}{|\Aut(N)|} \, b(N,G,G')$$

\noindent where $\Aut(G,G')$ denotes the group of automorphisms of $G$ taking $G'$ to $G'$, $\Aut(N)$ denotes the group of automorphisms of $N$ and $b(N,G,G')$ denotes the number of subgroups $G^*$ of $\Hol(N)$ such that there is an isomorphism from $G$ to $G^*$ taking $G'$ to the stabilizer in $G^*$ of $1_N$.

\end{corollary}

Recently a relationship has been found between Hopf Galois structures and an algebraic structure called brace. Classical braces were introduced by W. Rump \cite{R}, as a generalisation of radical rings, in order to study the non-degenerate involutive set-theoretic solutions of the quantum Yang-Baxter equation. Recently, skew braces were introduced by Guarnieri and Vendramin \cite{G-V} in order to study the non-degenerate (not necessarily involutive) set-theoretic solutions. This connection is further exploited in \cite{SV}, where the relation of braces with other algebraic structures is established.

\begin{definition} A skew (left) brace is a set $B$ endowed with two binary operations $\cdot$ and $\circ$ such that $(B,\cdot)$ and $(B,\circ)$ are groups (not necessarily abelian) and the two operations are related by the skew brace property

$$ a\circ(b\cdot c)=(a\circ b)\cdot a^{-1} \cdot (a\circ c), \text{\ for all \ } a,b,c \in B,$$

\noindent
where $a^{-1}$ denotes the inverse of $a$ in $(B,\cdot)$. The groups $(B,\cdot)$ and $(B,\circ)$ are called respectively the additive group and the multiplicative group of the skew brace $B$. If the additive group of $B$ is abelian, we call $B$ a (classical) brace.

A map between skew braces is a skew brace morphism if it is a group morphism both between the additive and the multiplicative groups.

\end{definition}

The relation between braces and Hopf-Galois structures was first proved by Bachiller for classical braces (see \cite{Ba} Proposition 2.3) and generalized by Guarnieri and Vendramin to skew braces.

\begin{proposition}[\cite{G-V} Proposition 4.3] \label{GV}
Let $(N,\cdot)$ be a group. There is a bijective correspondence between isomorphism classes of left skew braces with additive group isomorphic to $(N,\cdot)$ and classes of regular subgroups of $\Hol(N)$ under conjugation by elements of $\Aut(N)$.
\end{proposition}

We recall that, if $G$ is a regular subgroup of $\Hol(N)$, then $\pi:G \rightarrow N, (x,\varphi) \mapsto x$ is bijective. Then $N$ with the operation

$$x \circ y= \pi(\pi^{-1}(x) \pi^{-1}(y))=x \varphi(y)$$

\noindent
is a group isomorphic to $G$ and $(N,\cdot,\circ)$ is a skew left brace.

To a skew left brace one associates its socle, annihilator and its group of automorphisms.

\begin{definition} Let $(B,\cdot,\circ)$ be a skew left brace. We define its socle $\Soc(B)$ by

$$\Soc(B):= \{ a \in B \, | \, a\circ b = ab, b(b\circ a)=(b\circ a) b , \, \forall b \in B \},$$

\noindent
and its annihilator $\Ann(B)$ by

$$\Ann(B):=\Soc(B) \cap Z(B,\circ), $$

\noindent
where $\Z(B,\circ)$ is the centre of the group $(B,\circ)$.

\end{definition}

\begin{proposition}[\cite{SV} Example 4.3] \label{soc} Let $(B,\cdot,\circ)$ be a skew left brace. Then $\Soc(B)$ is contained in the center $Z(B,\cdot)$ of $(B,\cdot)$. Moreover it is a normal subgroup of $(B,\circ)$.
\end{proposition}

An automorphism of a skew brace $B$ is a bijection from $B$ to $B$ which is a group morphism with respect to both operations in $B$. If $B$ is the skew brace associated to a regular subgroup $G$ of $\Hol(N)$, we have the following characterization of the group of automorphisms $\Aut(B)$ of $B$ (see \cite{Z}, formula (3) below Proposition 2.5).

$$\Aut(B) \simeq \{ \varphi \in \Aut(N): \varphi G \varphi^{-1} \subset G \}.$$

\vspace{0.2cm}
Bachiller \cite{Ba1} classified braces of order $p^3$, for a prime
number $p$. Nejabati Zenouz \cite{Z, ZT} classified skew left braces of order $p^3$, for a prime number $p$. Dietzel \cite{D} classified skew left braces of order $p^2q$, for $p$ and $q$ prime numbers with $p<q-1$.  Catino, Colazzo and Stefanelli \cite{CCS} presented a method to determine skew left braces with non-trivial annihilator.
In \cite{CS3} we determined exactly the possible sets of Hopf Galois structure types for separable field extensions of degree $2p^2$.
In this paper, we consider the groups $N$ of order $2p^2$, for $p$ an odd prime number. For such a group, we determine all regular subgroups of the holomorph $\Hol(N)$. This result leads on the one hand to the determination of the Hopf Galois structures on a Galois field extension of degree $2p^2$. On the other hand it allows to classify the skew left braces of order $2p^2$. We note that we obtain in particular skew braces with trivial annihilator as for instance all those with additive group the nonabelian group of order $2p^2$ with noncyclic $p$-Sylow group and trivial center. Moreover we determine the separable field extensions of degree $2p^2$ having a Hopf Galois
structure of cyclic type and the number of these structures. In \cite{AB1} and \cite{AB2} Acri and Bonatto enumerate the skew left braces of size $p^2q$, for $p, q$ primes. In \cite{CCC} Campedel, Caranti and Del Corso classify the Hopf-Galois structures on Galois extensions of
degree $p^2q$, such that the Sylow p-subgroups of the Galois group are cyclic, and the corresponding skew braces. In \cite{AB} Alabdali and Byott determine the number of skew left braces of square free size.

\section{Groups of order $2p^2$, for $p$ an odd prime number}\label{groups}

Let $p$ denote an odd prime. As seen in \cite{CS3}, there are five groups of order $2p^2$, up to isomorphism. These are $C_{2p^2}$, $D_{2p^2}$, $C_p \times C_{2p}$, $C_p \times D_{2p}$ and $(C_p\times C_p) \rtimes C_2$. The automorphism group for each of them was determined in \cite{CS3}. Let us recall them.

\begin{enumerate}[1)]
\item $\Aut(C_{2p^2}) \simeq (\Z/2p^2 \Z)^*$ is cyclic of order $p(p-1)$.
\item For $G=D_{2p^2}=\langle r,s|r^{p^2}=s^2=1,srs=r^{-1} \rangle$, an automorphism is given by $r\mapsto r^i, s \mapsto r^j s$, with $0\leq i,j\leq p^2-1, p\nmid i$. We have then $|\Aut(G)|=(p^2-p)p^2=p^3(p-1)$.
\item For $G=C_p\times C_p \times C_2=\langle a \rangle \times \langle b \rangle \times \langle c \rangle$, $c$ is the unique element of order 2. An element in $\Aut(G)$ is then given by $a \mapsto a^i b^j, b\mapsto a^k b^l, c \mapsto c$, with $0\leq i,j,k,l \leq p-1, p\nmid il-jk$. We have then $\Aut(G) \simeq \GL(2,p)$ and $|\Aut(G)|=(p^2-1)(p^2-p)=p(p+1)(p-1)^2$.
\item For $G=C_p\times D_{2p}$, let $C_p=\langle c \rangle$ and $D_{2p}=\langle r,s|r^p=s^2=1,srs=r^{-1} \rangle$. An automorphism of $G$ is given by $c \mapsto c^k, r \mapsto r^i, s \mapsto r^j s$, with $1\leq i,k\leq p-1, 0\leq j \leq p-1$. We have then $|\Aut(G)|=p(p-1)^2$.
\item For $G=(C_p\times C_p) \rtimes C_2$, we write $C_p\times C_p=\langle a \rangle \times \langle b \rangle$ and $C_2=\langle c \rangle$. We have $cac=a^{-1}$ and $cbc=b^{-1}$. An automorphism of $G$ is given by $a \mapsto a^ib^j, b \mapsto a^k b^l, c \mapsto a^mb^n c$, with $0\leq i,j,k,l,m,n\leq p-1, p\nmid il-jk$. We have then $|\Aut(G)|=(p^2-1)(p^2-p)p^2=p^3(p+1)(p-1)^2$.

\end{enumerate}

\section{Galois extensions of degree $2p^2$}\label{Galois}

In this section we consider a Galois field extension $L/K$ of degree $2p^2$ and determine the number of Hopf Galois structures on $L/K$ for each possible type. We shall prove the following theorem throughout this section.

\begin{theorem}\label{number}
Let $L/K$ be a Galois field extension of degree $2p^2$, with $p$ an odd prime number. Let $G=\Gal(L/K)$ and let $N$ be a group of order $2p^2$. Then the number of Hopf Galois structures on $L/K$ of type $N$ is as given in the following table.

\begin{center}
\begin{tabular}{|c||c|c|c|c|c|}
\hline
Galois group $G$ $\diagdown$ type $N$ &  $C_{2p^2}$& $D_{2p^2}$ & $C_p\times C_{2p}$ & $C_p \times D_{2p}$ & $(C_p\times C_p) \rtimes C_2$ \\
\hline
\hline
$C_{2p^2}$ & $p$  & $2p$ & $0$ & $0$ & $0$ \\
\hline
$D_{2p^2}$ &   $p^2$  & $2$ & $0$ & $0$ & $0$ \\
\hline
$C_p\times C_{2p}$ & $0$ & $0$ & $p^2$ & $2p(p+1)$ & $p(3p+1)$ \\
\hline
$C_p\times D_{2p}$ & $0$ & $0$ & $p^2$ & $2p(p+1)$ & $p(3p+1)$ \\
\hline
$(C_p\times C_p)\rtimes C_2$ & $0$ & $0$ & $p^2$ & $2p^2(p+1)$  & $2p^3+p^2-p+2$ \\

\hline
\end{tabular}
\end{center}

\end{theorem}

\begin{lemma}\label{equiv} The statement in Theorem \ref{number} is equivalent to the number of transitive subgroups of $\Hol(N)$ isomorphic to $G$ being as given in the following table.

\begin{center}
\begin{tabular}{|c||c|c|c|c|c|}
\hline
$G$ $\diagdown$ $N$ &  $C_{2p^2}$& $D_{2p^2}$ & $C_p\times C_{2p}$ & $C_p \times D_{2p}$ & $(C_p\times C_p) \rtimes C_2$ \\
\hline
\hline
$C_{2p^2}$ & $p$  & $2p^3$ & $0$ & $0$ & $0$ \\
\hline
$D_{2p^2}$ &   $1$  & $2$ & $0$ & $0$ & $0$ \\
\hline
$C_p\times C_{2p}$ & $0$ & $0$ & $p^2$ & $2p$ & $p^3(3p+1)$ \\
\hline
$C_p\times D_{2p}$ & $0$ & $0$ & $p^2(p+1)$ & $2p(p+1)$ & $p^3(p+1)(3p+1)$ \\
\hline
$(C_p\times C_p)\rtimes C_2$ & $0$ & $0$ & $1$ & $2$  & $2p^3+p^2-p+2$ \\

\hline
\end{tabular}
\end{center}
\end{lemma}

\begin{proof} The result follows by applying Corollary \ref{cor} and the determination of the automorphism groups of the groups of order $2p^2$ given in Section \ref{groups}.
\end{proof}

Applying Theorem 3 in \cite{CS3} and the fact that the only groups of order $p^2$ are $C_{p^2}$ and $C_p\times C_p$, we obtain that, for a given group $N$, with $p$-Sylow subgroup $\Syl_p(N)$, the transitive subgroups $G$ of $\Hol(N)$ of order $2p^2$ have a $p$-Sylow subgroup $\Syl_p(G)$ isomorphic to $\Syl_p(N)$. This gives all zeros in the table in Lemma \ref{equiv}. Moreover $\Syl_p(G)$ is a semiregular subgroup of $\Hol(N)$. In each case, given $N$, we shall determine first the semiregular subgroups of $\Hol(N)$ isomorphic to $\Syl_p(N)$ and afterwards the elements of order 2 normalizing those subgroups and generating together a regular subgroup of $\Hol(N)$. We will then obtain explicitly all regular subgroups of $\Hol(N)$. This will prove the validity of the data in the table in Lemma \ref{equiv}, column by column.

\subsection{Regular subgroups of $\Hol(C_{2p^2})$}\label{ciclic}

\begin{lemma}\label{le11} $\Hol(C_{2p^2})$ has exactly $p$ subgroups of order $p^2$, all of them acting on $C_{2p^2}$ without fixed points.
\end{lemma}

\begin{proof}
If an element $(x,\varphi) \in \Hol(C_{2p^2})=C_{2p^2}\rtimes \Aut(C_{2p^2})$ has order $p^2$, then $\varphi^{p} = \Id$, since $\Aut(C_{2p^2})\simeq (\Z/2p^2\Z)^*$ is cyclic of order $p(p-1)$. Now, if $\varphi$ has order $p$, $\varphi(x)=x^i$, for $i$ of order $p$ modulo $p^2$. We have then

$$(x,\varphi)^{p}=(x\varphi(x)\dots \varphi^{p-1}(x),\varphi^{p})=(x^{1+i+\dots+i^{p-1}},\Id).$$

\noindent Now $1+i+\dots+i^{p-1}=\dfrac{1-i^{p}}{1-i}$ is divisible by $p$ but not by $p^2$. Indeed, we have $i\equiv \lambda p +1 \pmod{p^2}$, for some $\lambda$ with $1\leq \lambda \leq p-1$ and $i^p-1 \equiv \lambda p^2 \pmod{p^3}$. Hence $(x,\varphi)$ has order $p^2$. We have then that $(x,\varphi)$ has order $p^2$ if and only if $\varphi^p=\Id$ and $x$ has order $p^2$. Since $C_{2p^2}$ has $p(p-1)$ elements of order $p^2$ and $\Aut(C_{2p^2})$ has $p-1$ elements of order $p$, we obtain $p^2(p-1)$ elements of order $p^2$ in $\Hol(C_{2p^2})$. Since a group of order $p^2$ has $p(p-1)$ generators, we obtain that $\Hol(C_{2p^2})$ has exactly $p$ subgroups of order $p^2$, generated by $(x,\varphi)$, with $x$ of order $p^2$ and $\varphi^p=\Id$. Clearly, all of them act on $C_{2p^2}$ without fixed points.
\end{proof}

\begin{proposition}\label{prop1} The regular subgroups of $\Hol(C_{2p^2})$ are precisely $p$ regular subgroups isomorphic to $C_{2p^2}$ and $1$ regular subgroup isomorphic to $D_{2p^2}$.
\end{proposition}

\begin{proof} We look for elements of order 2 in $\Hol(C_{2p^2})$ which normalize the semiregular subgroups found in Lemma \ref{le11}.
If an element $(z,\chi) \in \Hol(C_{2p^2})$ has order $2$, then $\chi^{2} = \Id$. The only element $\chi$ of order 2 in $\Aut(C_{2p^2})$ satisfies $\chi(x)=x^{-1}$.
Now $(z,\chi)(x,\varphi)(z,\chi)=(z\chi(x)(\chi\varphi)(z),\varphi)$.

We have $\langle (x,\varphi),(z,\chi) \rangle \simeq C_{2p^2}$ if and only if $z\chi(x)(\chi\varphi)(z)=x$. This implies $\chi=\Id$ and $z$ of order 2. We obtain then $p$ subgroups of $\Hol(C_{2p^2})$ isomorphic to $C_{2p^2}$ of the form $\langle(x,\varphi)\rangle$, with $\varphi^p=\Id$ and $x$ a generator of $C_{2p^2}$, which are clearly regular.

In order to have $\langle (x,\varphi),(z,\chi) \rangle \simeq D_{2p^2}$, we need $\varphi=\Id$ and $z\chi(x)\chi(z)=x^{-1}$. We have then $\chi(x)=x^{-1}$ and we obtain a unique regular subgroup of $\Hol(C_{2p^2})$ isomorphic to $D_{2p^2}$, namely $\langle (a^2,\Id),(a,\chi) \rangle$, for $a$ a generator of $C_{2p^2}$, $\chi$ the element of order 2 in $\Aut(C_{2p^2})$.
\end{proof}

\subsection{Regular subgroups of $\Hol(D_{2p^2})$}\label{dihedral}

Let us write $D_{2p^2}=\langle r,s|r^{p^2}=s^2=1,srs=r^{-1} \rangle$.

\begin{lemma} Let us consider the automorphisms $\varphi_1$ and $\varphi_2$ of $D_{2p^2}$ defined by

$$\begin{array}{llll} \varphi_1:& r & \mapsto & r \\ & s &\mapsto & rs \end{array}, \quad \begin{array}{llll} \varphi_2:& r & \mapsto & r^{p+1} \\ & s &\mapsto & s \end{array}.$$

\noindent The elements $r, \varphi_1, \varphi_2$ generate the only $p$-Sylow subgroup of $\Hol(D_{2p^2})$.
\end{lemma}

\begin{proof} We check that $\varphi_1$ has order $p^2$, $\varphi_2$ has order $p$ and $\varphi_2\varphi_1=\varphi_1^{p+1} \varphi_2$, hence $\langle \varphi_1,\varphi_2 \rangle$ is a $p$-Sylow subgroup of $\Aut(D_{2p^2})$ isomorphic to $G_p$, the non-abelian group of order $p^3$ having exponent $p^2$. Moreover, since $|\Aut(D_{2p^2})|=p^3(p-1)$ this $p$-Sylow subgroup is unique. Now $\langle r \rangle \rtimes \langle \varphi_1, \varphi_2 \rangle$ is the only $p$-Sylow subgroup of $\Hol(D_{2p^2})$.
\end{proof}

\begin{lemma} $\Hol(D_{2p^2})$ has $p^3-p^2$ semiregular cyclic subgroups of order $p^2$, namely

$$\langle(r,\varphi_1^j\varphi_2^k)\rangle, \quad 0\leq j <p^2, \, j\not \equiv -1 \, (\mo p), \, 0\leq k <p.$$
\end{lemma}

\vspace{0.2cm}
\begin{proof}
An element $(r^i,\varphi_1^j\varphi_2^k)$ of $\Hol(D_{2p^2})$ has order $p^2$ if and only if $p\nmid i$ or $p\nmid j$. We have then $p^5-p^3$ elements of order $p^2$ in $\Hol(D_{2p^2})$ and hence $p^3+p^2$ subgroups of order $p^2$. If $p\mid i$, then the subgroup $\langle (r^i,\varphi_1^j\varphi_2^k)\rangle$ is not semiregular since the orbit of $r$ under its action has at most $p$ elements. We are then left with the $p^3$ subgroups of the form $\langle(r,\varphi_1^j\varphi_2^k)\rangle$, $0\leq j <p^2, 0\leq k <p$.  Under the action of one of them, the orbit of 1 contains all powers of $r$ and the orbit of $s$ contains $p^2$ elements if and only if $p\nmid j+1$.
\end{proof}

\begin{proposition}\label{prop2} The regular subgroups of $\Hol(D_{2p^2})$ are precisely $2p^3$  regular subgroups isomorphic to $C_{2p^2}$ and $2$ regular subgroups isomorphic to $D_{2p^2}$.
\end{proposition}

\begin{proof}
We look first for elements $(z,\chi)$ of order 2. If $\chi=\Id$, we may take $z=r^k s, 0\leq k \leq p^2-1$. If $\chi$ has order 2, then $\chi=\chi_l$ defined by $\chi_l(r)=r^{-1}, \chi_l(s)=r^l s$, for some $l$ with $0\leq l\leq p^2-1$. For $z=r^ks$, we have $z\chi(z)=r^k sr^{-k} r^ls=r^{2k-l}$. Hence $(r^ks,\chi_l)$ has order 2 if $l=2k$.

We now determine when $(z,\chi)$ of order 2 normalizes $\langle (r,\varphi)\rangle$. Let us write $\varphi=\varphi_1^i \varphi_2^j$.
We have

$$(r^k s,\Id)(r,\varphi)(r^ks,\Id)=(r^k sr\varphi(r^ks),\varphi)=(r^{-kjp-i-1},\varphi).$$

\noindent
If $\varphi=\Id$, we obtain $\langle (r^ks,\Id),(r,\Id)  \rangle=\langle r,s \rangle \simeq D_{2p^2}$. If $\varphi$ has order $p$, then $p\mid i$ and $-kjp-i-1 \not \equiv 1 \pmod{p}$, hence $(r^k s,\Id)$ does not normalize $\langle (r,\varphi) \rangle$. If $\varphi$ has order $p^2$, then $(r^k s,\Id)$  normalizes $\langle (r,\varphi) \rangle$ for $i=-2-kjp$ and $\langle (r^ks, \Id),(r,\varphi) \rangle \simeq C_{2p^2}$. We obtain then $p^3$ transitive subgroups of $\Hol(D_{2p^2})$ isomorphic to $C_{2p^2}$.

If $\chi$ has order 2, then $\chi=\chi_{2k}$ defined by $\chi_{2k}(r)=r^{-1}, \chi_{2k}(s)=r^{2k} s$ and $z=r^ks$. We have

$$(r^ks,\chi_{2k})(r,\varphi)(r^ks,\chi_{2k})=(r^ksr^{-1}\chi_{2k}(\varphi(r^ks)),\chi_{2k}\varphi\chi_{2k}).$$

\noindent
Now $\chi_{2k}\varphi\chi_{2k}(r)=r^{jp+1}=\varphi(r)$ and $\chi_{2k}\varphi\chi_{2k}(s)=r^{-2kjp-i}s$. If $\varphi=\Id$, then we get $r^ksr^{-1}\chi_{2k}(\varphi(r^ks))=r$. We obtain then $p^2$ transitive subgroups of $\Hol(D_{2p^2})$ isomorphic to $C_{2p^2}$, namely $\langle (r,\Id),(r^ks,\chi_{2k})\rangle$, $0\leq k \leq p^2-1$.

If $\varphi$ has order $p$, then $p\mid i$. In this case, $\chi_{2k}\varphi\chi_{2k}=\varphi$ when $i=-kjp$ and then $r^ksr^{-1}\chi_{2k}(\varphi(r^ks))=r$. We obtain then $p^3-p^2$ transitive subgroups of $\Hol(D_{2p^2})$ isomorphic to $C_{2p^2}$, namely $\langle (r^ks,\chi_{2k}),(r,\varphi) \rangle$, with $\varphi=\varphi_1^{-kjp} \varphi_2^j$, $0\leq k \leq p^2-1, 1\leq j \leq p-1$.

Finally, if $\varphi$ has order $p^2$, then $p\nmid i$ and $j$ must be zero. We have then $\chi_{2k} \varphi \chi_{2k}=\varphi^{-1}$ and $r^ksr^{-1}\chi_{2k}(\varphi(r^ks))=r^{1+i}$. Hence $(r^ks,\chi_{2k})$ normalizes $(r,\varphi)$ for $i=-2$ and we have $\langle (r,\varphi_1^{-2}),(r^ks,\chi_{2k})\rangle \simeq D_{2p^2}$. Now $(r,\varphi_1^{-2})(r^ks,\chi_{2k})=(r^{k-1}s,\chi_{2(k-1)})$, hence these groups coincide for all $k$ and we obtain 1 transitive subgroup of $\Hol(D_{2p^2})$ isomorphic to $D_{2p^2}$.
\end{proof}

\subsection{Regular subgroups of $\Hol(C_p\times C_{2p})$}\label{abelian}

Let us write $C_p\times C_p \times C_2=\langle a \rangle \times \langle b \rangle \times \langle c \rangle$.

\begin{lemma}
$\Hol(C_p\times C_{2p})$ has  $p^2$ semiregular subgroups isomorphic to $C_p\times C_p$.
\end{lemma}

\begin{proof}
 We consider the element $\varphi$ in $\Aut (C_p\times C_{2p})$  given by $a \mapsto ab, b\mapsto b, c \mapsto c$. Then $\varphi$ has order $p$ and, in $\Hol(C_p\times C_{2p})$, we have the relation $(1,\varphi)(a,\Id)=(b,\Id)(a,\Id)(1,\varphi)$ whereas $(1,\varphi)$ and $(b,\Id)$ commute with each other. Hence $a,b,\varphi$ generate a $p$-Sylow subgroup $\Syl_p$ of $\Hol(C_p\times C_{2p})$ isomorphic to the Heisenberg group $H_p$. The normalizer of $\Syl_p$ in $\Hol(C_p\times C_{2p})$ consists in the elements $(z,\psi)$ such that $\psi$ normalizes $\langle \varphi \rangle$ in $\Aut(C_p\times C_{2p}) \simeq \GL(2,p)$, hence has index $p+1$ in $\Hol(C_p\times C_{2p})$. Now $\Syl_p$ has $p+1$ subgroups isomorphic to $C_p\times C_p$, namely

\begin{equation}\label{ABi}
A:=\langle a,b \rangle \text{\ and \ }B_i:=\langle (a^i,\varphi),b\rangle, 0\leq i \leq p-1,
\end{equation}

\noindent
which are semiregular, except for $\langle (1,\varphi),b\rangle$. Since $\langle a,b \rangle$ is normal in $\Hol(C_p\times C_{2p})$, we obtain $(p-1)(p+1)+1=p^2$ semiregular subgroups of $\Hol(C_p\times C_{2p})$ isomorphic to $C_p\times C_p$.
\end{proof}

\begin{proposition}\label{prop3} The regular subgroups of $\Hol(C_p\times C_{2p})$ are precisely $p^2$ regular subgroups isomorphic to $C_p\times C_{2p}$, $p^3+p^2$ regular subgroups isomorphic to $C_p\times D_{2p}$ and 1 regular subgroup isomorphic to $(C_p\times C_p) \rtimes C_2$.
\end{proposition}

\begin{proof}
We look for elements of order 2 normalizing $A$ or $B_i$ in (\ref{ABi}). If an element $(z,\chi) \in \Hol(C_{2p^2})$ has order $2$, then $\chi^{2} = \Id$. If $\chi=\Id$, then $z=c$, the only element of order 2 in $C_p \times C_{2p}$, and $(c,\Id)$ centralizes both $A$ and $B_i$. We obtain then $p^2$ regular subgroups of $\Hol(C_p\times C_{2p})$ isomorphic to $C_p\times C_{2p}$.

If $\chi$ has order 2, then $z\chi(z)=1$. We have

$$\begin{array}{l} (z,\chi)(a,\Id)(z,\chi)=(z\chi(a)\chi(z),\Id)=(\chi(a),\Id)\\
 (z,\chi)(b,\Id)(z,\chi)=(z\chi(b)\chi(z),\Id)=(\chi(b),\Id).
  \end{array}
  $$

\noindent Hence $(z,\chi)$ normalizes $A$. We obtain the regular subgroup $\langle a,b,(c,\chi) \rangle$ of $\Hol(C_p\times C_{2p})$ isomorphic to $(C_p\times C_p) \rtimes C_2$, when $\chi(a)=a^{-1}, \chi(b)=b^{-1}$, and $p(p+1)=p^2+p$ regular subgroups of $\Hol(C_p\times C_{2p})$ isomorphic to $C_p\times D_{2p}$, conjugated to the subgroup $\langle a,b,(c,\chi) \rangle$, where $\chi$ is defined by $\chi(a)=a^{-1}, \chi(b)=b$.

Now

$$(z,\chi)(a^i,\varphi)(z,\chi)=(z\chi(a)^i\chi\varphi(z),\chi\varphi\chi).$$

\noindent So that $(z,\chi)$ normalizes $B_i$, we must have $\chi(b) \in \langle b \rangle$ and $\chi\varphi\chi \in \langle \varphi \rangle$.

We have three possibilities for $\chi$, namely $\chi_1$ defined by  $\chi_1(a)=a^{-1}$ and $\chi_1(b)=b^{-1}$; $\chi_{2,j}$ defined by $\chi_{2,j}(a)=a^{-1}b^j$ and $\chi_{2,j}(b)=b$; $\chi_{3,j}$ defined by $\chi_{3,j}(a)=ab^j$ and $\chi_{3,j}(b)=b^{-1}$. Then $\chi\varphi\chi=\varphi$ in the first case and $\chi\varphi\chi=\varphi^{-1}$ in the other two. In order to obtain a regular subgroup, we take $z=a^kc$. Now $z\chi_{2,j}(z)=1$ implies $jk=0$ and $z\chi_{3,j}(z)=1$ implies $k=0$. We further impose  $(z,\chi)(a^i,\varphi)(z,\chi)=(z\chi(a)^i\chi\varphi(z),\chi\varphi\chi)$ to lie in $B_i$. We obtain the regular subgroups $B_i \rtimes \langle (a^kc,\chi_{2,j})\rangle$ isomorphic to $C_p\times D_{2p}$. Taking into account $\varphi \chi_{2,j}=\chi_{2,j-1}$, these groups may be written as $B_i \rtimes \langle (a^k c, \chi_{2,0})\rangle, 0<i\leq p-1, 0\leq k \leq p-1$. Together with their conjugates, these are $(p-1)p(p+1)=p^3-p$ subgroups, so we get in all $p^3+p^2$ regular subgroups of $\Hol(C_p\times C_{2p})$ isomorphic to $C_p\times D_{2p}$, as wanted.
\end{proof}

\subsection{Regular subgroups of $\Hol(C_p\times D_{2p})$}\label{cidi}

We write $C_p\times D_{2p}=\langle r,s,c \rangle$ as in Section \ref{groups}.

\begin{lemma} Let $\varphi$ denote the automorphism  of $C_p\times D_{2p}$ given by $r \mapsto r, s \mapsto  rs, c \mapsto c$. The $p^2-p$ subgroups

\begin{equation}\label{Akl} A_{k,l}:=\langle (r,\varphi^k),(c,\varphi^l) \rangle, \quad 0\leq k\leq p-2, \, 0\leq l\leq p-1,
\end{equation}

\noindent
isomorphic to $C_p\times C_p$, are precisely the semiregular subgroups of $\Hol(C_p\times D_{2p})$ of order $p^2$.
\end{lemma}

\begin{proof}
 The automorphism $\varphi$ has order $p$ and commutes with $r$ and $c$ in $\Hol(C_p\times D_{2p})$. Hence $r, c$ and $\varphi$ generate a $p$-Sylow subgroup $\Syl_p$ of $\Hol(C_p\times D_{2p})$ isomorphic to $C_p\times C_p \times C_p$. Moreover $\Syl_p$ is normal in $\Hol(C_p\times D_{2p})$, hence it is the unique $p$-Sylow subgroup. We obtain that $\Hol(C_p\times D_{2p})$ has $p^2+p+1$ subgroups isomorphic to $C_p\times C_p$. Now, for $x,y \in \langle r,c\rangle$, the orbit of 1 under the action of the subgroup $\langle (x,\varphi^k),(y,\varphi^l) \rangle$ has $p^2$ elements when $\langle x,y\rangle=\langle r,c\rangle$ and the subgroup may be written as $A_{k,l}:=\langle (r,\varphi^k),(c,\varphi^l) \rangle$. The orbit of $s$ has $p^2$ elements if and only if $p\nmid k+1$.  We obtain then that $\Hol(C_p\times D_{2p})$ has $p^2-p$ subgroups isomorphic to $C_p\times C_p$ and semiregular, which may be written as $A_{k,l}:=\langle (r,\varphi^k),(c,\varphi^l) \rangle$, with $0\leq k\leq p-2, 0\leq l\leq p-1$.
\end{proof}

\begin{proposition}\label{prop4} The regular subgroups of $\Hol(C_p\times D_{2p})$ are precisely $2p$ regular subgroups isomorphic to $C_p\times C_{2p}$, $2p^2+2p$ regular subgroups isomorphic to $C_p\times D_{2p}$ and $2$ regular subgroups isomorphic to $(C_p\times C_{p})\rtimes C_2$.
\end{proposition}

\begin{proof}
We look for elements $(z,\chi)$ of order 2 normalizing $A_{k,l}$ in (\ref{Akl}) and generating together a transitive group. If $\chi=\Id$, then $z=r^is$ and $(r^is,\Id)$ normalizes  $A_{p-2,0}$ and $A_{0,l}$. We obtain then the $p$ subgroups $A_{p-2,0}\times \langle(r^is,\Id)\rangle$, $0\leq i \leq p-1$, isomorphic to $C_p\times C_{2p}$ and the $p$ subgroups $A_{0,l}\rtimes \langle (s,\Id)\rangle$, $0\leq l \leq p-1$, isomorphic to $C_p\times D_{2p}$.

If $\chi$ has order 2, we must have $z\chi(z)=1$. We have three possibilities, namely $(r^is,\chi_{1,i})$, with $\chi_{1,i}(r)=r^{-1}, \chi_{1,i}(c)=c, \chi_{1,i}(s)=r^{2i} s, 1\leq i \leq p-1$; $(r^ic^js,\chi_{2,i})$, with $\chi_{2,i}(r)=r^{-1}, \chi_{2,i}(c)=c^{-1}, \chi_{2,i}(s)=r^{2i} s$, $0\leq i,j \leq p-1$; $(r^ic^js,\chi_3)$, with $\chi_3(r)=r, \chi_3(c)=c^{-1}, \chi_3(s)=s$, $0\leq i,j \leq p-1$. We obtain

$$(r^is,\chi_{1,i})(r,\varphi^k)(r^is,\chi_{1,i})=(r^{k+1},\varphi^{-k}), \quad (r^is,\chi_{1,i})(c,\varphi^l)(r^is,\chi_{1,i})=(cr^{l},\varphi^{-l}).$$

\noindent Then $(r^is,\chi_{1,i})$ normalizes $A_{p-2,l}$ and $A_{p-2,l} \rtimes \langle (r^is,\chi_{1,i})\rangle$ is isomorphic to $C_p\times D_{2p}$, $0\leq i,l \leq p-1$. Since $(r,\varphi^{-2})(r^is,\chi_{1,i})=(r^{i-1}s,\chi_{1,i-1})$, each of these groups contains $p$ elements of the form $(r^is,\chi_{1,i})$ and we have then $p$ regular subgroups isomorphic to $C_p\times D_{2p}$, which may be written as $\langle (r,\varphi^{-2}), (c,\varphi^l), (s,\chi_{1,0})\rangle, 0\leq l \leq p-1$. Also $(r^is,\chi_{1,i})$ normalizes $A_{0,0}$ and $A_{0,0} \times \langle (r^is,\chi_{1,i})\rangle$ is isomorphic to $C_p\times C_{2p}$, $0\leq i\leq p-1$. Now

$$(r^ic^js,\chi_{2,i})(r,\varphi^k)(r^ic^js,\chi_{2,i})=(r^{k+1},\varphi^{-k}), \quad (r^ic^js,\chi_{2,i})(c,\varphi^l)(r^ic^js,\chi_{2,i})=(c^{-1}r^{l},\varphi^{-l}).$$

\noindent The element $(r^{k+1},\varphi^{-k})$ belongs to $A_{k,l}$ for $k=0$ and $k=p-2$. In the first case, $(c^{-1}r^{l},\varphi^{-l})$ belongs to $A_{0,l}$ for any $l$ and in the second case it belongs to $A_{p-2,l}$ only for $l=0$. Then $(r^ic^js,\chi_{2,i})$  normalizes $A_{0,l}$ and $A_{0,l} \rtimes \langle (c^js,\chi_{2,i})\rangle$ is isomorphic to $C_p\times D_{2p}$. This amounts to $p^2$ regular subgroups isomorphic to $C_p\times D_{2p}$ since each of these groups contains $p$ elements of the form $(c^js,\chi_{2,i})$. We write them $\langle r,(c,\varphi^l),(s,\chi_{2,i})\rangle, 0\leq i,l \leq p-1$. Moreover $(r^ic^js,\chi_{2,i})$  normalizes $A_{p-2,0}$ and $A_{p-2,0} \rtimes \langle (r^is,\chi_{2,i})\rangle$ is isomorphic to $(C_p\times C_{p})\rtimes C_2$, $0\leq i \leq p-1$. In fact these groups are equal for all values of $i$, since $(r,\varphi^{-2})=(r^{i-1},\varphi^{-2}\chi_{2,i})$ and $\varphi^{-2}\chi_{2,i}=\chi_{2,i-1}$. We obtain then one transitive subgroup of $\Hol(C_p\times D_{2p})$ isomorphic to $(C_p\times C_{p})\rtimes C_2$. Finally

$$(r^ic^js,\chi_3)(r,\varphi^k)(r^ic^js,\chi_3)=(r^{-k-1},\varphi^{k}), \quad (r^ic^js,\chi_3)(c,\varphi^l)(r^ic^js,\chi_3)=(c^{-1}r^{-l},\varphi^{l}).$$

\noindent Then $(r^ic^js,\chi_3)$ normalizes $A_{p-2,l}$  and $A_{p-2,l} \rtimes \langle (r^ic^js,\chi_3)\rangle$ is isomorphic to $C_p\times D_{2p}$. Now $A_{p-2,l}=\langle (cr^{l/2},\Id),(r,\varphi^{-2}) \rangle$, so we have $p^2$ different subgroups, namely,

$$\langle (cr^{l/2},\Id),(r,\varphi^{-2}),(r^is,\chi_3) \rangle, 0\leq i,l\leq p-1.$$

\noindent
Moreover $(r^ic^js,\chi_3)$ normalizes $A_{0,0}$  and $\langle r,c\rangle \rtimes \langle (s,\chi_3)\rangle$ is isomorphic to $(C_p\times C_p) \rtimes C_2$.
\end{proof}

\subsection{Regular subgroups of $\Hol((C_p\times C_{p})\rtimes C_2)$}\label{semi}

Let us write $(C_p\times C_{p})\rtimes C_2=(\langle a \rangle \times \langle b \rangle)\rtimes \langle c \rangle$ as in Section \ref{groups}.

\begin{lemma}\label{le51} Let us consider the automorphisms $\varphi_1, \varphi_2, \varphi_3$ of $(C_p\times C_{p})\rtimes C_2$ given by

$$\begin{array}{cccc}\varphi_1:&a & \mapsto & a \\ &b& \mapsto & ab \\ &c&\mapsto & c \end{array}, \quad \begin{array}{cccc}\varphi_2:&a & \mapsto & a \\ &b& \mapsto & b \\ &c&\mapsto & ac \end{array},\quad \begin{array}{cccc}\varphi_3:&a & \mapsto & a \\ &b& \mapsto & b \\ &c&\mapsto & bc \end{array}.$$

We have
\begin{enumerate}[1)]
\item $\Syl_p:=\langle a,b \rangle \rtimes  \langle \varphi_1, \varphi_3\rangle$ is a $p$-Sylow subgroup of $\Hol((C_p\times C_{p})\rtimes C_2)$.
\item The number of $p$-Sylow subgroups of $\Hol((C_p\times C_{p})\rtimes C_2)$ is $p+1$.
\item The subgroup $F:=\langle a,b,\varphi_2,\varphi_3 \rangle$ is isomorphic to $C_p^4$ and is normal in $\Hol((C_p\times C_{p})\rtimes C_2)$, hence contained in all $p$-Sylow subgroups of $\Hol((C_p\times C_{p})\rtimes C_2)$.
\end{enumerate}
\end{lemma}

\begin{proof}

\noindent
1) The automorphisms $\varphi_1, \varphi_2, \varphi_3$ have order $p$ and satisfy $\varphi_1\varphi_3=\varphi_2\varphi_3\varphi_1$ whereas $\varphi_2$ commutes with both $\varphi_1$ and $\varphi_3$. Hence $\langle \varphi_1, \varphi_2, \varphi_3\rangle=\langle \varphi_1, \varphi_3\rangle$ is a $p$-Sylow subgroup of $\Aut((C_p\times C_{p})\rtimes C_2)$ isomorphic to the Heisenberg group $H_p$. Therefore $\Syl_p=\langle a,b \rangle \rtimes  \langle \varphi_1, \varphi_3\rangle$ is a $p$-Sylow subgroup of $\Hol((C_p\times C_{p})\rtimes C_2)$.

\noindent
2) We shall determine the normalizer of $\Syl_p$ in $\Hol((C_p\times C_{p})\rtimes C_2)$. Since $\langle a,b \rangle$ is normal in $\langle a,b,c \rangle$, so that $(x,\varphi)$ normalizes $\Syl_p$, it is enough that $\varphi$ normalizes $\langle \varphi_1, \varphi_3\rangle$. We consider $\varphi$ defined by

\begin{equation}\label{fi}
\varphi(a)= a^ib^j, \varphi(b)= a^k b^l, \varphi(c)= a^mb^n c, \quad 0\leq i,j,k,l,m,n\leq p-1, \, p\nmid il-jk.
\end{equation}

\noindent
Then $\varphi^{-1}$ is given by $a \mapsto a^{i'}b^{j'}, b \mapsto a^{k'} b^{l'}, c \mapsto a^{-(mi'+nk')}b^{-(mj'+nl')} c$, where $\left(\begin{smallmatrix}i'&j'\\k'&l' \end{smallmatrix} \right)=\left(\begin{smallmatrix}i&j\\k&l \end{smallmatrix} \right)^{-1}$ in $\GL(2,p)$. We have $(\varphi\varphi_1\varphi^{-1})(a)=a^{1+ij'}b^{jj'}, (\varphi\varphi_1\varphi^{-1})(b)=a^{il'}b^{1+jl'}$, hence $\varphi\varphi_1\varphi^{-1} \in \langle \varphi_1,\varphi_2,\varphi_3 \rangle$ if and only if $j=j'=0$. Now clearly $\varphi\varphi_3\varphi^{-1} \in \langle \varphi_2, \varphi_3 \rangle$, for any $\varphi$. We obtain then that the normalizer of $\Syl_p$ in $\Hol((C_p\times C_{p})\rtimes C_2)$ has order $p^3(p-1)^2$, hence $\Hol((C_p\times C_{p})\rtimes C_2)$ has $p+1$ $p$-Sylow subgroups.

\noindent
3) Taking into account that $\varphi_2$ commutes with $\varphi_3$ and the action of $\varphi_2$ and $\varphi_3$ on $a$ and $b$, we obtain that the subgroup $F=\langle a,b,\varphi_2,\varphi_3 \rangle$ is isomorphic to $C_p^4$. Moreover, for any $\varphi$ as in (\ref{fi}), we have $\varphi\varphi_2\varphi^{-1}, \varphi\varphi_3\varphi^{-1} \in \langle \varphi_2, \varphi_3 \rangle$, hence $F$ is normal in $\Hol((C_p\times C_{p})\rtimes C_2)$.
\end{proof}

\begin{lemma} With $F$ and $\Syl_p$ as in Lemma \ref{le51}, we have

\begin{enumerate}[1)]
\item The $p(p-1)^2(p+1)$ subgroups

$$\langle (a,\varphi_2^i\varphi_3^j),(b,\varphi_2^k\varphi_3^l)\rangle, \quad 0\leq i,j,k,l\leq p-1, \, p\nmid (i+1)(l+1)-jk$$

\noindent
are precisely the semiregular subgroups of $F$ isomorphic to $C_p \times C_p$.
\item The $p(p-1)^3$ subgroups

$$\langle (a,\varphi_2^{i_2}),(b,\varphi_1^{j_1}\varphi_2^{j_2} \varphi_3^{j_3}) \rangle, \quad 0\leq j_2 \leq p-1, 0\leq i_2, j_3 \leq p-2,1\leq j_1 \leq p-1$$

\noindent
are precisely the semiregular subgroups of $\Syl_p$ isomorphic to $C_p \times C_p$, not contained in $F$.
\item  The number of semiregular subgroups of $\Hol((C_p\times C_p)\rtimes C_2)$ isomorphic to $C_p\times C_p$ is equal to $p^2(p-1)^2(p+1)$.
\end{enumerate}
\end{lemma}

\begin{proof}

\noindent
1) A subgroup of $F$ isomorphic to $C_p\times C_p$ is generated by two nontrivial elements $(x,\varphi),(y,\psi)$ such that $(y,\psi)\not \in \langle (x,\varphi)\rangle$. If $\langle x,y \rangle=\langle a,b \rangle$, then the orbit of $1$ is the whole group $\langle a,b \rangle$. The subgroups satisfying this condition may be written as

$$\langle (a,\varphi_2^i\varphi_3^j),(b,\varphi_2^k\varphi_3^l)\rangle, 0\leq i,j,k,l\leq p-1.$$

\noindent Now the orbit of $c$ contains $p^2$ elements if and only if $p\nmid (i+1)(l+1)-jk$. We obtain then $(p^2-1)(p^2-p)=p(p-1)^2(p+1)$ semiregular subgroups of $F$ isomorphic to $C_p\times C_p$.

\noindent
2) We look now for semiregular subgroups of $\Syl_p$ isomorphic to $C_p\times C_p$ and not contained in $F$. Such a subgroup is generated by two mutually commuting elements $(a,\varphi_1^{i_1}\varphi_2^{i_2}\varphi_3^{i_3})$ and $(b,\varphi_1^{j_1}\varphi_2^{j_2}\varphi_3^{j_3})$, with $i_1$ and $j_1$ not both zero. Now, since $\varphi_3^{i_3} \varphi_1^{j_1}=\varphi_2^{-j_1i_3} \varphi_1^{j_1} \varphi_3^{i_3}$, we have that $\varphi_1^{i_1}\varphi_2^{i_2} \varphi_3^{i_3}$ commutes with $\varphi_1^{j_1} \varphi_2^{j_2}\varphi_3^{j_3}$ exactly when $i_1j_3\equiv i_3j_1 \pmod{p}$. Under this condition $(a,\varphi_1^{i_1}\varphi_2^{i_2} \varphi_3^{i_3})$ and $(b,\varphi_1^{j_1} \varphi_2^{j_2}\varphi_3^{j_3})$ commute if and only if $a(\varphi_1^{i_1} \varphi_2^{i_2}\varphi_3^{i_3})(b)=b(\varphi_1^{j_1}\varphi_2^{j_2} \varphi_3^{j_3})(a)$, which gives $a^{i_1+1}b=ab$, that is $i_1=0$. We obtain then the subgroups

$$\langle (a,\varphi_2^{i_2}),(b,\varphi_1^{j_1}\varphi_2^{j_2} \varphi_3^{j_3}) \rangle, 0\leq i_2, j_2, j_3 \leq p-1, 1\leq j_1 \leq p-1,$$

\noindent under whose action the orbit of 1 is the whole group $\langle a, b \rangle$. Now the orbit of $c$ contains $p^2$ elements if and only if $i_2 \neq p-1$ and $j_3 \neq p-1$. We obtain then $p(p-1)^3$ semiregular subgroups of $\Syl_p$ isomorphic to $C_p\times C_p$ and not contained in $F$.

\noindent
3) Since the number of $p$-Sylow subgroups of $\Hol((C_p\times C_p)\rtimes C_2)$ is $p+1$, we obtain that the number of semiregular subgroups isomorphic to $C_p\times C_p$ is equal to $p(p-1)^3(p+1)+p(p-1)^2(p+1)=p^2(p-1)^2(p+1)$.
\end{proof}

We look now for elements of order 2 normalizing the subgroups above and generating together a regular subgroup of $\Hol((C_p\times C_p) \rtimes C_2)$. If $(z,\chi)$ has order 2, then $\chi^2=\Id$. If $\chi=\Id$, then $(z,\chi)=(a^mb^nc,\Id), 0\leq m,n \leq p-1$. If $\chi$ has order 2, then either $\chi=\chi_1$ given by

\begin{equation}\label{chi1}
 \chi_1(a)=a^rb^s, \chi_1(b)=a^tb^{-r}, \chi_1(c)=a^u b^v c, \quad \text{with \ }r^2+st=1, \, (r+1)u+tv=0, \, su+(1-r)v=0,
\end{equation}

\noindent
or $\chi=\chi_2$ given by

\begin{equation}\label{chi2}
\chi_2(a)= a^{-1}, \chi_2(b)= b^{-1}, \chi_2(c)=a^u b^v c,
\end{equation}

\noindent
with any $u,v$. Now, $(z,\chi_1)$ has order 2 for $z=a^mb^nc$, with $(1-r)m-tn=u$ and $(1+r)n-sm=v$; $(z,\chi_2)$ has order 2 for $z=a^{u/2}b^{v/2}c$.

\begin{lemma} The regular subgroups of $\Hol((C_p\times C_{p})\rtimes C_2)$, having an element of order 2 of the form $(a^mb^nc,\Id)$, are precisely $p^4$ regular subgroups isomorphic to $C_p\times C_{2p}$, $p^3(p+1)$ regular subgroups isomorphic to $C_p\times D_{2p}$ and $1$ regular subgroup isomorphic to $(C_p\times C_{p})\rtimes C_2$.
\end{lemma}

\begin{proof} We look first for regular subgroups with $p$-Sylow subgroup equal to $F$. We have

$$\begin{array}{l} (a^mb^nc,\Id)(a,\varphi_2^i\varphi_3^j)(a^mb^nc,\Id)=(a^{-i-1}b^{-j},\varphi_2^i\varphi_3^j), \\ (a^mb^nc,\Id)(b,\varphi_2^k\varphi_3^l)(a^mb^nc,\Id)=(a^{-k}b^{-l-1},\varphi_2^k\varphi_3^l). \end{array}$$

\noindent
The elements $(a^{-i-1}b^{-j},\varphi_2^i\varphi_3^j)$ and $(a^{-k}b^{-l-1},\varphi_2^k\varphi_3^l)$ belong to $\langle (a,\varphi_2^i\varphi_3^j),(b,\varphi_2^k\varphi_3^l)\rangle$ if and only if $i,j,k,l$ satisfy the system

$$\left\{ \begin{array}{l} -i^2-i-jk=i \\ -ij-j-jl=j \Leftrightarrow j(-i-l-2)=0\\-ik-kl-k=k\Leftrightarrow k(-i-l-2)=0\\ -jk-l^2-l=l \end{array} \right.$$

\noindent If $j=k=0$, then $(i,l)=(0,0),(0,p-2),(p-2,0)$ or $(p-2,p-2)$.
We obtain the $p^2$ subgroups $\langle (a,\varphi_2^{-2}),(b,\varphi_3^{-2}),(a^mb^nc,\Id)\rangle$ isomorphic to $C_p\times C_{2p}$;  the $p$ subgroups $\langle (a,\varphi_2^{-2}),(b,\Id),(a^mc,\Id)\rangle$ and the $p$ subgroups $\langle (a,\Id),(b,\varphi_3^{-2}),(b^nc,\Id)\rangle$ isomorphic to $C_p\times D_{2p}$; and the subgroup $\langle a,b,c \rangle$ isomorphic to $(C_p\times C_p)\rtimes C_2$.
If $i+l=-2$, then $kj=-i^2-2i$ and with $i,j,k,l$ satisfying these relations we further obtain  $p(p+2)(p-1)$ subgroups isomorphic to $C_p\times D_{2p}$. Here we take into account that for $i=0$ or $i=-2$, we have $j=0$ or $k=0$ and that each subgroup isomorphic to $C_p\times D_{2p}$ contains $p$ elements of order 2.

We look now for regular subgroups with $p$-subgroup contained in $\Syl_p$ and different from $F$. We have

$$\begin{array}{rcl} (a^mb^nc,\Id)(a,\varphi_2^{i_2})(a^mb^nc,\Id)&=&(a^{-i_2-1},\varphi_2^{i_2})
), \\ (a^mb^nc,\Id)(b,\varphi_1^{j_1}\varphi_2^{j_2} \varphi_3^{j_3})(a^mb^nc,\Id)&=& (a^{-j_2-nj_1-j_1j_3}b^{-j_3-1},\varphi_1^{j_1}\varphi_2^{j_2} \varphi_3^{j_3}).\end{array}$$

\noindent The elements $(a^{-i_2-1},\varphi_2^{i_2})$ and $(a^{-j_2-nj_1-j_1j_3}b^{-j_3-1},\varphi_1^{j_1}\varphi_2^{j_2} \varphi_3^{j_3})$ belong to $\langle (a,\varphi_2^{i_2}),(b,\varphi_1^{j_1}\varphi_2^{j_2} \varphi_3^{j_3})\rangle$ if and only if either $i_2=-2$, $j_3=-2$ and $j_2=(2-n)j_1$ or $i_2=0$ and $j_3=-2$.

We obtain the $p^2(p-1)$ subgroups $\langle (a,\varphi_2^{-2}),(b,\varphi_1^{j_1}\varphi_2^{(2-n)j_1}\varphi_3^{-2}),(a^mb^nc,\Id)\rangle$ isomorphic to $C_p\times C_{2p}$, the $p^2(p-1)$ subgroups $\langle (a,\Id),(b,\varphi_1^{j_1}\varphi_2^{j_2} \varphi_3^{-2}),(b^nc,\Id)\rangle$ isomorphic to $C_p\times D_{2p}$ and no subgroup isomorphic to $(C_p\times C_p)\rtimes C_2$.

Summing up, taking into account that the number of $p$-Sylow subgroups of $\Hol((C_p\times C_p) \rtimes C_2)$ is $p+1$, we obtain $p^2+p^2(p^2-1)=p^4$ regular subgroups isomorphic to $C_p\times C_{2p}$, $2p+p(p+2)(p-1)+p^2(p^2-1)=p^3(p+1)$ regular subgroups isomorphic to $C_p\times D_{2p}$ and $1$ regular subgroup isomorphic to $(C_p\times C_{p})\rtimes C_2$.
\end{proof}

\begin{lemma} The regular subgroups of $\Hol((C_p\times C_{p})\rtimes C_2)$, having an element of order 2 of the form $(z,\chi_1)$ and with $p$-Sylow subgroup equal to $F$, are precisely $p^3(p+1)$ regular subgroups isomorphic to $C_p\times C_{2p}$, $p^3(p+1)^2$ regular subgroups isomorphic to $C_p\times D_{2p}$ and $p(p+1)$ regular subgroups isomorphic to $(C_p\times C_{p})\rtimes C_2$.
\end{lemma}

\begin{proof}
We have $\chi_1\varphi_2\chi_1=\varphi_2^r \varphi_3^s$, $\chi_1\varphi_3\chi_1=\varphi_2^t \varphi_3^{-r}$ and

$$\begin{array}{l} (a^mb^nc,\chi_1)(a,\varphi_2^i\varphi_3^j)(a^mb^nc,\chi_1)=(a^{-r-ir-tj}b^{-s-is+rj},\varphi_2^{ri+tj}\varphi_3^{si-rj}), \\ (a^mb^nc,\chi_1)(b,\varphi_2^k\varphi_3^l)(a^mb^nc,\chi_1)=(a^{-t-kr-tl}b^{r-ks+rl},\varphi_2^{rk+tl}\varphi_3^{sk-rl}). \end{array}$$

 By imposing the elements on the right to be equal to $(a,\varphi_2^i\varphi_3^j)$ and $(b,\varphi_2^k\varphi_3^l)$, respectively, we obtain the $p^3$ subgroups $\langle (a,\Id),(b,\varphi_2^k\varphi_3^{-2}),(a^mb^n c,\chi_1) \rangle$, with $r=-1, s=0, t=-k, u=2m+kn, v=0$, the $p^3$ subgroups $\langle (a,\varphi_2^{-2}),(b,\varphi_2^k),(a^m b^nc,\chi_1) \rangle$, with $r=1, s=0, t=-k, u=kn, v=2n$ and the $p^3(p-1)$ subgroups  $\langle (a,\varphi_2^i\varphi_3^j),(b,\varphi_2^k\varphi_3^l),(a^mb^nc,\chi_1) \rangle$, with $j\neq 0$, $jk=-i^2-2i,l=-i-2$, $r=-i-1, s=-j, t=-k, ju=(2+i)v,jm=v+in$, all of them isomorphic to $C_p\times C_{2p}$.

 By imposing the elements on the right to be equal to $(a,\varphi_2^i\varphi_3^j)^{-1}=(a^{-1},\varphi_2^{-i}\varphi_3^{-j})$ and $(b,\varphi_2^k\varphi_3^l)$, respectively, we obtain the $p^2$ subgroups $\langle (a,\Id),(b,\varphi_2^k),(b^n c,\chi_1) \rangle$, with $r=1, s=0, t=-k, u=kv/2, n=v/2$, isomorphic to $C_p \times D_{2p}$. We further obtain the $p^2$ subgroups $\langle (a,\varphi_2^{-2}),(b,\varphi_2^k\varphi_3^{-2}),(a^mb^n c,\chi_1) \rangle$, with $r=-1, s=0, t=-k, v=0, m=(u-kn)/2$, and the $p^2(p-1)$ subgroups $\langle (a,\varphi_2^{i}\varphi_3^j),(b,\varphi_2^k\varphi_3^{i}),(a^mb^n c,\chi_1) \rangle$, with $j\neq 0, jk=i^2+2i, r=i+1, s=j, t=-k, ju=iv, jm=(i+2)n-v$, all of them isomorphic to $C_p \times D_{2p}$, where we are taking into account that each of these subgroups contains $p$ elements of order $2$ of the form $(a^mb^n c,\chi_1)$. Each of these subgroups isomorphic to $C_p \times D_{2p}$ has $p(p+1)$ different conjugates generated by elements of the form $(a,\varphi_2^i\varphi_3^j), (b,\varphi_2^k\varphi_3^l)$ and $(z,\chi_1)$.

Finally by imposing the elements on the right to be equal to $(a,\varphi_2^i\varphi_3^j)^{-1}=(a^{-1},\varphi_2^{-i}\varphi_3^{-j})$ and $(b,\varphi_2^k\varphi_3^l)^{-1}=(b^{-1},\varphi_2^{-k}\varphi_3^{-l})$, respectively, we obtain the $p(p+1)$ subgroups \linebreak $\langle (a,\varphi_2^{i}\varphi_3^j),(b,\varphi_2^k\varphi_3^{-i-2}),(c,\chi_1) \rangle$, with $jk=-i^2-2i, r=i+1, s=j, t=k, u=v=0$, all of them isomorphic to $(C_p \times C_{p})\rtimes C_2$. In order to determine the number of subgroups, we take into account that $(C_p\times C_p) \rtimes C_2$ contains $p^2$ elements of order 2.

Summing up, we obtain $p^3(p+1)$ regular subgroups isomorphic to $C_p\times C_{2p}$, $p^3(p+1)^2$ regular subgroups isomorphic to $C_p\times D_{2p}$ and $p(p+1)$ regular subgroups isomorphic to $(C_p\times C_{p}) \rtimes C_2$.
\end{proof}

\begin{lemma} The regular subgroups of $\Hol((C_p\times C_p)\rtimes C_2)$, having an element of order 2 of the form $(z,\chi_1)$ and with $p$-Sylow subgroup different from $F$ are precisely $2p^3(p^2-1)$ regular subgroups isomorphic to $C_p\times D_{2p}$ and $2p(p^2-1)$ regular subgroups isomorphic to $(C_p\times C_{p})\rtimes C_2$.
\end{lemma}

\begin{proof}
We have $\chi_1\varphi_1\chi_1(a)=a^{1+rs}b^{s^2}, \chi_1\varphi_1\chi_1(b)=a^{-r^2}b^{1-rs}, \chi_1\varphi_1\chi_1(c)=a^{rv}b^{sv} c$. Hence $\chi_1\varphi_1\chi_1 \in \langle \varphi_1,\varphi_2,\varphi_3 \rangle$ implies $s=0$ and then $\chi_1\varphi_1\chi_1=\varphi_1^{-r^2} \varphi_2^{rv}$. Moreover $s=0$ implies $r=\pm 1$, by (\ref{chi1}). Now, with $s=0$, we have

$$(a^mb^nc,\chi_1)(b,\varphi_1^{j_1}\varphi_2^{j_2}\varphi_3^{j_3})(a^mb^nc,\chi_1)=(a^{-t-rj_2-rn-rj_3-tj_3}b^{r+rj_3},
\varphi_1^{-r^2j_1}\varphi_2^{rvj_1+rj_2+tj_3}\varphi_3^{-rj_3})$$

\noindent
and

$$(a^mb^nc,\chi_1)(a,\varphi_2^{i_2})(a^mb^nc,\chi_1)=(a^{-r-ri_2},\varphi_2^{ri_2}).$$

\noindent This last element belongs to the subgroup $\langle (a,\varphi_2^{i_2}),(b,\varphi_1^{j_1}\varphi_2^{j_2} \varphi_3^{j_3}) \rangle$ if and only if $i_2=0$ or $i_2=-2$. Let us note that $(b,\varphi_1^{j_1}\varphi_2^{j_2}\varphi_3^{j_3})^{-1}=(a^{j_1} b^{-1},\varphi_1^{-j_1}\varphi_2^{-j_2-j_1j_3}\varphi_3^{-j_3})$. We distinguish now the two cases $r=1$ and $r=-1$.

If $r=1$,

\begin{equation}
(a^mb^nc,\chi_1)(b,\varphi_1^{j_1}\varphi_2^{j_2}\varphi_3^{j_3})(a^mb^nc,\chi_1)=(a^{-t-j_2-n-j_3-tj_3}b^{1+j_3},
\varphi_1^{-j_1}\varphi_2^{vj_1+j_2+tj_3}\varphi_3^{-j_3}) \label{eqf1}
\end{equation}

\noindent
Now, if $i_2=0$, then $(a^mb^nc,\chi_1)(a,\Id)(a^mb^nc,\chi_1)=(a^{-1},\Id)$. Hence $(a^mb^nc,\chi_1)$ normalizes $\langle (a,\Id),(b,\varphi_1^{j_1}\varphi_2^{j_2} \varphi_3^{j_3}) \rangle$ if and only if $j_3=-2, (v-2)j_1+2j_2-2t=0$. We obtain then the $p(p-1)$ subgroups $\langle (a,\Id),(b,\varphi_1^{j_1}\varphi_2^{j_2} \varphi_3^{-2}),(b^nc,\chi_1) \rangle$,  with $j_1 \neq 0, j_1n=j_1-j_2, r=1, s=t=0, u=0, v=2n$, isomorphic to $(C_p\times C_p) \rtimes C_2$. Now if $i_2=-2$, then $(a^mb^nc,\chi_1)(a,\varphi_2^{-2})(a^mb^nc,\chi_1)=(a,\varphi_2^{-2})$.  The element in the righthand side of (\ref{eqf1}) belongs to the subgroup $\langle (a,\varphi_2^{-2}),(b,\varphi_1^{j_1}\varphi_2^{j_2} \varphi_3^{j_3}) \rangle$ if and only if $j_3=-2$ and $4j_1+2n-4-vj_1=0$. Taking into account the relation $v=2n$, we obtain $(j_1-1)(2-n)=0$. We have the further relation $u=-tn$. This gives the $p^3$ transitive subgroups $\langle (a,\varphi_2^{-2}),(b,\varphi_1\varphi_2^{j_2}\varphi_3^{-2}),(a^mb^nc,\chi_1)\rangle$, with $u=-tn, v=2n$ and the $(p-2)p^3$ transitive subgroups $\langle (a,\varphi_2^{-2}),(b,\varphi_1^{j_1}\varphi_2^{j_2}\varphi_3^{-2}),(a^mb^2c,\chi_1)\rangle$, with $j_1 \neq 1, u=-2t, v=4$, all of them isomorphic to $D_{2p} \times C_p$.

If $r=-1$,

\begin{equation}
(a^mb^nc,\chi_1)(b,\varphi_1^{j_1}\varphi_2^{j_2}\varphi_3^{j_3})(a^mb^nc,\chi_1)=(a^{-t+j_2+n+j_3-tj_3}b^{-1-j_3},
\varphi_1^{-j_1}\varphi_2^{-vj_1-j_2+tj_3}\varphi_3^{j_3}). \label{eqf2}
\end{equation}

\noindent
Now, if $i_2=0$, then  $(a^mb^nc,\chi_1)(a,\Id)(a^mb^nc,\chi_1)=(a,\Id)$. The element in the righthand side of (\ref{eqf2}) belongs to the subgroup $\langle (a,\Id),(b,\varphi_1^{j_1}\varphi_2^{j_2} \varphi_3^{j_3}) \rangle$ if and only if $j_3=0$ and $v=0$. We obtain the $(p-1)p^3$ transitive groups $\langle (a,\Id), (b,\varphi_1^{j_1}\varphi_2^{j_2}),(c,\chi_1)\rangle$, with $v=0$, isomorphic to $C_p\times D_{2p}$.
 Now if $i_2=-2$, then $(a^mb^nc,\chi_1)(a,\varphi_2^{-2})(a^mb^nc,\chi_1)=(a^{-1},\varphi_2^{2})$.  Hence $(a^mb^nc,\chi_1)$ normalizes $\langle (a,\varphi_2^{-2}),(b,\varphi_1^{j_1}\varphi_2^{j_2} \varphi_3^{j_3}) \rangle$ if and only if $j_3=0, v=0$ and $n=t-j_2$. We obtain then the $p(p-1)$ transitive subgroups $\langle (a,\varphi_2^{-2}),(b,\varphi_1^{j_1}\varphi_2^{j_2}),(c,\chi_1) \rangle$, with $r=-1, s=0, t=j_2, u=v=0$, isomorphic to $(C_p\times C_p)\rtimes C_2$.

Summing up, we obtain the following numbers of regular subgroups with $p$-Sylow subgroup contained in $\Syl_p$ and different from $F$: $2p^3(p-1)$ isomorphic to $C_p\times D_{2p}$ and $2p(p-1)$ isomorphic to $(C_p\times C_{p})\rtimes C_2$.
The corresponding number of regular subgroups of $\Hol((C_p\times C_{p})\rtimes C_2)$ is obtained taking into account that the number of $p$-Sylow subgroups of $\Hol((C_p\times C_{p})\rtimes C_2)$ is $p+1$.
\end{proof}

\begin{lemma} The regular subgroups of $\Hol((C_p\times C_{p})\rtimes C_2)$, having an element of order 2 of the form $(z,\chi_2)$ are precisely $p^4$ regular subgroups isomorphic to $C_p\times C_{2p}$, $p^3(p+1)$ regular subgroups isomorphic to $C_p\times D_{2p}$ and $1$ regular subgroup isomorphic to $(C_p\times C_{p})\rtimes C_2$.
\end{lemma}

\begin{proof}
We have $\chi_2\varphi_2\chi_2=\varphi_2^{-1}$, $\chi_2\varphi_3\chi_2=\varphi_3^{-1}$ and

$$\begin{array}{l} (a^{u/2}b^{v/2}c,\chi_2)(a,\varphi_2^i\varphi_3^j)(a^{u/2}b^{v/2}c,\chi_2)=(a^{i+1}b^{j},\varphi_2^{-i}\varphi_3^{-j}), \\ (a^{u/2}b^{v/2}c,\chi_2)(b,\varphi_2^k\varphi_3^l)(a^{u/2}b^{v/2}c,\chi_2)=(a^{k}b^{l+1},\varphi_2^{-k}\varphi_3^{-l}). \end{array}$$

By imposing the elements on the right to be equal to $(a,\varphi_2^i\varphi_3^j)$ and $(b,\varphi_2^k\varphi_3^l)$ respectively, we obtain the $p^2$ subgroups $\langle (a,\Id),(b,\Id), (a^{u/2}b^{v/2} c,\chi_2) \rangle$ isomorphic to $C_p\times C_{2p}$. By imposing the elements on the right to be equal to $(a,\varphi_2^i\varphi_3^j)^{-1}=(a^{-1},\varphi_2^{-i}\varphi_3^{-j})$ and $(b,\varphi_2^k\varphi_3^l)$ respectively, we obtain the $p$ subgroups $\langle (a,\varphi_2^{-2}),(b,\Id), (b^{v/2} c,\chi_2) \rangle$ isomorphic to $C_p\times D_{2p}$. Each of these groups has $p(p+1)$ conjugates. By imposing the elements on the right to be equal to $(a,\varphi_2^i\varphi_3^j)^{-1}=(a^{-1},\varphi_2^{-i}\varphi_3^{-j})$ and $(b,\varphi_2^k\varphi_3^l)^{-1}=(b^{-1},\varphi_2^{-k}\varphi_3^{-l})$, we obtain the subgroup $\langle (a,\varphi_2^{-2}),(b,\varphi_3^{-2}), (c,\chi_2) \rangle$, with $u=v=0$,  isomorphic to $(C_p\times C_p) \rtimes C_2$.

Now we have $\chi_2\varphi_1\chi_2=\varphi_2^{-v}\varphi_1$ and

$$\begin{array}{l} (a^{u/2}b^{v/2}c,\chi_2)(a,\varphi_2^{i_2})(a^{u/2}b^{v/2}c,\chi_2)=(a^{i_2+1},\varphi_2^{-i_2}), \\ (a^{u/2}b^{v/2}c,\chi_2)(b,\varphi_1^{j_1}\varphi_2^{j_2}\varphi_3^{j_3})(a^{u/2}b^{v/2}c,\chi_2)=(a^{vj_1/2+j_2+j_1j_3}b^{j_3+1},\varphi_1^{j_1}\varphi_2^{-vj_1-j_2}
\varphi_3^{-j_3}). \end{array}$$

Taking into account the equality

$$(b,\varphi_1^{j_1}\varphi_2^{j_2}\varphi_3^{j_3})^k=(a^{j_1k(k-1)/2}b^k,\varphi_1^{kj_1}\varphi_2^{kj_2-j_1j_3k(k-1)/2}\varphi_3^{kj_3}),$$

\noindent
we obtain that the elements on the right belong to $\langle (a,\varphi_2^{i_2}),(b,\varphi_1^{j_1}\varphi_2^{j_2}\varphi_3^{j_3})\rangle$ only if $i_2^2+2i_2=0$ and $j_3=0$. If $i_2=j_3=0$, we have the further condition $vj_1+2j_2=0$. We obtain the $p^2(p-1)$ subgroups $\langle (a,\Id),(b,\varphi_1^{j_1}\varphi_2^{-vj_1/2}),(a^{u/2}b^{v/2}c,\chi_2)\rangle$ isomorphic to $C_p\times C_{2p}$. If $i_2=-2$ and $j_3=0$, there is no extra condition and we obtain the $p^2(p-1)$ subgroups $\langle (a,\varphi_2^{-2}),(b,\varphi_1^{j_1}\varphi_2^{j_2}),(b^{v/2}c,\chi_2)\rangle$ isomorphic to $C_p\times D_{2p}$.

Summing up, taking into account that the number of $p$-Sylow subgroups of $Hol((C_p\times C_p) \rtimes C_2$ is $p+1$, we obtain $p^2+p^2(p^2-1)=p^4$ regular subgroups isomorphic to $C_p\times C_{2p}$, $p^2(p+1)+p^2(p^2-1)=p^3(p+1)$ regular subgroups isomorphic to $C_p\times D_{2p}$ and $1$ regular subgroup isomorphic to $(C_p\times C_{p})\rtimes C_2$.
\end{proof}

Summing up the results in lemmas 21 to 24, we obtain

\begin{proposition}\label{prop5}
The regular subgroups of $\Hol((C_p\times C_p)\rtimes C_2)$ are precisely $p^3(3p+1)$ regular subgroups isomorphic to $C_p\times C_{2p}$, $p^3(p+1)(3p+1)$ regular subgroups isomorphic to $C_p\times D_{2p}$ and $2p^3+p^2-p+2$ regular subgroups isomorphic to $(C_p\times C_p) \rtimes C_2$.
\end{proposition}

\noindent
{\it End of proof of Theorem \ref{number}.} Propositions \ref{prop1}, \ref{prop2}, \ref{prop3}, \ref{prop4} and \ref{prop5} prove column by column the correctness of the table in Lemma \ref{equiv}, hence using this Lemma, the Theorem is proved. \hfill $\Box$

\section{Hopf Galois structures of cyclic type}

In this section we determine the separable extensions of degree $2p^2$ having a Hopf Galois structure of cyclic type and the number of these structures. The Galois case has been already studied in Section \ref{Galois}.

\begin{theorem} Let $L/K$ be a separable non Galois field extension of degree $2p^2$, for $p$ an odd prime number. Let $\wL$ be a normal closure of $L/K$, $G=\Gal(\wL/K)$. Then $L/K$ has a Hopf Galois structure of cyclic type if and only if $G$ is isomorphic to one of the following groups:
\begin{enumerate}[1)]
\item the semidirect product $C_{2p^2}\rtimes C_{(p-1)/d}$ of a cyclic group of order $2p^2$ and a cyclic group of order $(p-1)/d$ for $d$ a proper divisor of $p-1$;
\item the semidirect product $C_{p^2}\rtimes C_{(p-1)/d}$ of a cyclic group of order $p^2$ and a cyclic group of order $(p-1)/d$ for $d$ a proper divisor of $(p-1)/2$;
\item the semidirect product $C_{2p^2}\rtimes C_{p(p-1)/d}$ of a cyclic group of order $2p^2$ and a cyclic group of order $p(p-1)/d$ for $d$ a divisor of $p-1$;
\item the semidirect product $C_{p^2}\rtimes C_{p(p-1)/d}$ of a cyclic group of order $p^2$ and a cyclic group of order $p(p-1)/d$ for $d$ a divisor of $(p-1)/2$.
\end{enumerate}
The number of structures is $p$ when $|G|= 2p^3$ and 1 in all other cases.
\end{theorem}

\begin{proof}
By Theorem \ref{theoB}, $L/K$ has a Hopf Galois structure of cyclic type if and only if $G=\Gal(\wL/K)$ embeds in $\Hol(C_{2p^2})$ as a transitive subgroup. If $G^* \subset \Hol(C_{2p^2})$ acts transitively on $C_{2p^2}$, then $2p^2$ divides the order of $G^*$. Since the $p$-Sylow subgroup of $\Hol(C_{2p^2})$ has order $p^3$, we have two possibilities for the $p$-Sylow subgroup $\Syl_p(G^*)$ of $G^*$, either $|\Syl_p(G^*)|=p^2$ or $\Syl_p(G^*)=\Syl_p(\Hol(C_{2p^2}))$. Let us denote by $a$ a generator of $C_{2p^2}$ and by $\varphi$ a generator of $\Aut(C_{2p^2})$. Then $\varphi(a)=a^i$, with $i$ an integer prime with $2p$ and having multiplicative order $p(p-1)$ modulo $2p^2$. We have $\Syl_p(\Hol(C_{2p^2}))=\langle a^2, \varphi^{p-1} \rangle.$ By Theorem 3 and the proof of Theorem 5 in \cite{CS3}, we have that a transitive subgroup of $\Hol(C_{2p^2})$ has an element of order $p^2$. We have seen in Section \ref{ciclic} that $\Hol(C_{2p^2})$ has precisely $p$ cyclic subgroups of order $p^2$. These may be written as $\langle (a^2,\varphi^{k(p-1)}) \rangle, 0\leq k \leq p-1$. Then either $\Syl_p(G^*)=\langle a^2, \varphi^{p-1} \rangle$ or $\Syl_p(G^*)=\langle (a^2,\varphi^{k(p-1)}) \rangle$, for some $k$. In both cases, since $[\Hol(C_{2p^2}):\Syl_p(\Hol(C_{2p^2}))]=p-1$, $\Syl_p(G^*)$ is a normal subgroup of $G^*$. We examine now the two cases.

\begin{enumerate}[I)]
\item We assume that $\Syl_p(G^*)$ has order $p^2$. If $k\neq 0$, then the normalizer of $\langle (a^2,\varphi^{k(p-1)}) \rangle$ in $\Hol(C_{2p^2})$ consists of the elements of the form $(a^j,\varphi^{l(p-1)})$, which have order a multiple of $p$ if $p\nmid l$. We have then two subcases.
\begin{enumerate}[1)]
\item If $a \in G^*$, then $G^*$ is transitive and we have $G^*=\langle a, \varphi^{pd} \rangle = \langle a \rangle \rtimes \langle \varphi^{pd} \rangle \simeq C_{2p^2} \rtimes C_{(p-1)/d}$, for some divisor $d$ of $p-1$.
\item If $a \not \in G^*$ and $|G^*|>2p^2$, then $a^2 \in G^*$. In order to be transitive, $G^*$ must contain an element of the form $(a,\varphi^k)$, for some $k$. Such an element has order prime with $p$ if and only if $p\mid k$. We have then $G^*=\langle a^2, (a,\varphi^{pd}) \rangle = \langle a^2 \rangle \rtimes \langle (a,\varphi^{pd}) \rangle \simeq C_{p^2} \rtimes C_{(p-1)/d}$, for some divisor $d$ of $(p-1)/2$.

\end{enumerate}

\item We assume that $\Syl_p(G^*)$ has order $p^3$. We have then $\Syl_p(G^*)=\langle a^2, \varphi^{p-1} \rangle.$ We have two subcases.
\begin{enumerate}
\item[3)] If $a \in G^*$, then $G^*$ is transitive and we have $G^*=\langle a, \varphi^{p-1},\varphi^{pd} \rangle=\langle a,\varphi^{d} \rangle \simeq C_{2p^2} \rtimes C_{p(p-1)/d}$, for $d$ a divisor of $p-1$.
\item[4)] If $a \not \in G^*$, in order to be transitive, $G^*$ must contain an element of the form $(a,\varphi^{pd})$ of even order. We have then $G^*=\langle a^2, \varphi^{p-1},(a,\varphi^{pd}) \rangle$, for some divisor $d$ of $(p-1)/2$. Now $G^*=\langle a^2, \varphi^{p-1},(a^p,\varphi^{pd}) \rangle$ and the two elements $\varphi^{p-1}$ and $(a^p,\varphi^{pd})$ commute with each other, hence generate a cyclic subgroup of order $p(p-1)/d$. We have then $G^* =\langle a^2,(a^p,\varphi^d) \rangle \simeq C_{p^2} \rtimes C_{p(p-1)/d}$.

\end{enumerate}
\end{enumerate}

We apply now Corollary \ref{cor} in order to determine the number of Hopf Galois structures. In the sequel, we identify $G$ with $G^*$ and $G'$ with $\Stab(G^*,1)$. We consider the different cases.

\begin{enumerate}[1)]
\item If $G=\langle a, \varphi^{pd} \rangle$, then $G'=\langle \varphi^{pd} \rangle$ and, since $\varphi^{pd}(a)=a^{i^{pd}}$, an automorphism of $G$ sending $G'$ to itself must send the element $\varphi^{pd}$ to itself, hence $|\Aut(G,G')|=|\Aut(C_{2p^2})|$ and there is one structure.

\item If $G=\langle a^2, (a,\varphi^{pd}) \rangle$, then $G'=\langle \varphi^{2pd} \rangle$ and again an automorphism of $G$ sending $G'$ to itself must send the element $\varphi^{2pd}$ to itself. We consider the subgroup of $\Aut(G,G')$ consisting of the automorphisms $g$ such that $g(a^2)=a^{2j}$, with $p\nmid j$ and $g(a,\varphi^{pd})=(a,\varphi^{pd})$. It has order $p(p-1)=|\Aut(C_{2p^2})|$. Now, an element $h$ in $\Aut(G,G')$ such that $h(a^2)=a^2$ must satisfy $h(a,\varphi^{pd})=(a,\varphi^{pdl})$, with $l=1$ or $l=1+(p-1)/(2d)$ in order to satisfy $h(\varphi^{2pd})=\varphi^{2pd}$. Now, since $\varphi^{pd}(a)=a^{i^{pd}}$ and $\varphi^{pdl}(a)=a^{-i^{pd}}$, for $l=1+(p-1)/(2d)$, the only possibility is $l=1$. Hence $|\Aut(G,G')|=|\Aut(C_{2p^2})|$
    and there is one structure.

\item If $G=\langle a,\varphi^{d} \rangle$, then $G'=\langle \varphi^{d} \rangle$ and, as in the other cases, an automorphism of $G$ sending $G'$ to itself must send the element $\varphi^{d}$ to itself. We consider the subgroup of $\Aut(G,G')$ consisting of the automorphisms $g$ such that $g(a)=a^{j}$, with $p\nmid j$ and $g(\varphi^{d})=\varphi^{d}$. It has order $p(p-1)=|\Aut(C_{2p^2})|$. A coset of this subgroup in $\Aut(G,G')$ has a representative $h$ such that $h(a)=(a,\varphi^{k(p-1)})$, for some integer $k$, $0\leq k \leq p-1$. Now we have $(1,\varphi^d)(a,\varphi^{k(p-1)})(1,\varphi^{-d})=(a^{i^d},\varphi^{k(p-1)})$. This last element is equal to $(a,\varphi^{k(p-1)})^{i^d}$, for $k\neq 0$, if and only if $i^d \equiv 1 \pmod{p}$ if and only if $d=p-1$. We have then that for $|G|=2p^3$, $|\Aut(G,G')|=p|\Aut(C_{2p^2})|$ and, in other cases $|\Aut(G,G')|=|\Aut(C_{2p^2})|$. The number of structures is then $p$ when $|G|=2p^3$ and 1 otherwise.

\item If $G=\langle a^2,(a^p,\varphi^d) \rangle$, taking into account the action of $(a^p,\varphi^d)$ on $a^2$, we obtain that an automorphism $g$ of $G$ satisfies $g(a^2)=a^{2j}$, for some $j$ with $p\nmid j$, and $g(a^p,\varphi^d)=(a^{k},\varphi^d)$, for some odd $k$. Now
    $G'=\langle \varphi^{2d} \rangle$ and, as in the other cases, an automorphism of $G$ sending $G'$ to itself must send the element $\varphi^{2d}$ to itself.
    Since $\varphi^{2d}=a^{-p(1+i^d)}(a^p,\varphi^d)^2$, the condition $g(\varphi^{2d})=\varphi^{2d}$ implies $p^2\mid (-jp+k)(1+i^d)$. We have $p\mid 1+i^d$ if and only if $d=(p-1)/2$. We obtain then that, if $d \neq (p-1)/2$, an automorphism in $\Aut(G,G')$ satisfies  $g(a^p,\varphi^d)=(a^{jp},\varphi^d)$, and then $|\Aut(G,G')|=|\Aut(C_{2p^2})|$ and the number of structures is 1. If $d=(p-1)/2$, an automorphism in $\Aut(G,G')$ satisfies  $g(a^p,\varphi^d)=(a^{k},\varphi^d)$, with $p\mid k$, and then $|\Aut(G,G')|=p|\Aut(C_{2p^2})|$ and the number of structures is $p$.
\end{enumerate}
\end{proof}

\section{Skew braces of order $2p^2$}\label{braces}

In this section, we will classify skew braces of order $2p^2$ by applying Proposition \ref{GV} to the regular subgroups of $\Hol(N)$ obtained in Section \ref{Galois}. For each skew brace $B$ corresponding to a pair $(N,G)$, where $N$ is a group of order $2p^2$ and $G$ a regular subgroup of $\Hol(N)$, modulo conjugation by $\Aut(N)$, we shall determine the socle $\Soc(B)$, the annihilator $\Ann(B)$ and the group of automorphisms $\Aut(B)$.

\subsection{Cyclic additive group}

Let $N$ be the cyclic group of order $2p^2$. We consider the regular subgroups of $\Hol(N)$ obtained in Section \ref{ciclic}. Let $\varphi$ be a fixed element of order $p$ in $\Aut(N)$ and $x$ be a generator of $N$. Then the subgroup $\langle (x,\varphi) \rangle$ of $\Hol(N)$ contains exactly $p$ elements of the form $(*,\varphi)$, namely $(x,\varphi)^{kp+1}, 0\leq k \leq p-1$. On the other hand, all elements of the form $(*,\varphi)$ are conjugated by $\Aut(N)$. We obtain then that all subgroups $\langle (x,\varphi) \rangle$, with $x$ a generator of $N$ are conjugated by $\Aut(N)$. We obtain then two braces with multiplicative group isomorphic to $C_{2p^2}$, corresponding to $G=\langle(x,\Id)\rangle=N$ and $G=\langle(x,\varphi)\rangle$, with $\varphi$ of order $p$. In the first case $\Soc(B)=\Ann(B)=B$ and $\Aut(B)=\Aut(N)$. In the second case $\Soc(B)=\Ann(B)$ is the subgroup of order $2p$ of $N$ and $\Aut(B)=\Aut(N)$.

The unique regular subgroup of $\Hol(N)$ isomorphic to $D_{2p^2}$ provides a skew brace $B$ with $\Soc(B)$ equal to the subgroup of order $p^2$ of $N$, $\Ann(B)=\{1\}$ and $\Aut(B)=\Aut(N)$.

We summarize the obtained results in the following table.

\begin{center}
\begin{tabular}{|c||c|c|c|c|}
\hline
\multicolumn{5}{|c|}{Braces of order $2p^2$ with cyclic additive group}\\
\hline \hline
number of braces & $(B,\circ)$  &  $|\Soc(B)|$ & $|\Ann(B)|$ & $|\Aut(B)|$ \\
\hline
\hline
1&$C_{2p^2}$ & $2p^2$  & $2p^2$ & $p(p-1)$  \\
\hline
1&$C_{2p^2}$ &   $2p$  & $2p$ & $p(p-1)$ \\
\hline \hline
1&$D_{2p^2}$ & $p^2$ & $1$ & $p(p-1)$  \\
\hline
\end{tabular}
\end{center}

\subsection{Dihedral additive group}

Let $N$ be the dihedral group of order $2p^2$. We consider the regular subgroups of $\Hol(N)$ obtained in Section \ref{dihedral} and use the notations introduced there. We have two regular subgroups of $\Hol(N)$ isomorphic to $D_{2p^2}$, namely $N$ and its centralizer $N^*$ in $\Hol(N)$. Since they are normal in $\Hol(N)$ they give two different skew braces with dihedral multiplicative group. In both cases we have $\Soc(B)=\Ann(B)=\{1\}$ and $\Aut(B)=\Aut(N)$.

We consider now the cyclic regular subgroups of $\Hol(N)$. For the $p^3$ subgroups $G_{k,j}:=\langle (r^ks,\Id),(r,\varphi)\rangle$, where $\varphi=\varphi_1^i \varphi_2^j$, with $i=-2-kjp$, we obtain $\varphi_1 G_{k,j} \varphi_1^{-1}=G_{k+1,j}$. For $\psi \in \Aut(N)$, defined by $\psi(r)=r^m, \psi(s)=s$, we have $\psi^{-1} G_{0,j} \psi = G_{0,mj}$. Hence the subgroups $G_{k,j}$ are distributed in two classes under conjugation by $\Aut (N)$, corresponding to $j=0$ and $j\neq 0$. We obtain then $2$ different skew braces. We have $\Soc(B)=\Ann(B)=\{1\}$ and $\Aut(B)=\{\varphi \in \Aut(N): \varphi(s)=s \}$, when $j=0$, hence $|\Aut(B)|=p^2-p$; $\Aut(B)=\{\varphi \in \Aut(N): \varphi(s)=s, \varphi(r)=r^m, m\equiv 1 \pmod{p} \}$, when $j\neq 0$, hence $|\Aut(B)|=p$.

For the $p^2$ subgroups $G_{k}:=\langle (r,\Id),(r^ks,\chi_{2k})\rangle$, where $\chi_{2k}(r)=r^{-1}, \chi_{2k}(s)=r^{2k}s$, we obtain $\varphi_1G_k \varphi_1^{-1}=G_{k+1}$. Hence they provide one skew brace. We have $\Soc(B)=\Ann(B)=\{1\}$ and $\Aut(B)=\{\varphi \in \Aut(N): \varphi(s)=s \}$, hence $|\Aut(B)|=p^2-p$.

For the $p^3-p^2$ subgroups $G_{k,j}:=\langle (r^ks,\chi_{2k}),(r,\varphi)\rangle$, where $\varphi=\varphi_1^{-kjp} \varphi_2^j$, with $p\nmid j$, we obtain $\varphi_1 G_{k,j} \varphi_1^{-1}=G_{k+1,j}$. For $\psi \in \Aut(N)$, defined by $\psi(r)=r^m, \psi(s)=s$, we have $\psi^{-1} G_{0,j} \psi = G_{0,mj}$. Hence the subgroups $G_{k,j}$ lie in one class under conjugation by $\Aut(N)$, and we obtain $1$  skew brace. We have $\Soc(B)=\Ann(B)=\{1\}$ and $\Aut(B)=\{\varphi \in \Aut(N): \varphi(s)=s, \varphi(r)=r^m, m\equiv 1 \pmod{p} \}$, hence $|\Aut(B)|=p$.

We summarize the obtained results in the following table.

\begin{center}
\begin{tabular}{|c||c|c|c|c|}
\hline
\multicolumn{5}{|c|}{Skew braces of order $2p^2$ with dihedral additive group}\\
\hline \hline
number of braces & $(B,\circ)$  &  $|\Soc(B)|$ & $|\Ann(B)|$ & $|\Aut(B)|$ \\
\hline
\hline
2 & $C_{2p^2}$ & $1$  & $1$ & $p(p-1)$  \\
\hline
2 & $C_{2p^2}$ &   $1$  & $1$ & $p$ \\
\hline \hline
2 & $D_{2p^2}$ & $1$ & $1$ & $p^3(p-1)$  \\
\hline
\end{tabular}
\end{center}

\subsection{Additive group isomorphic to $C_p\times C_{2p}$}

Let $N$ be the group $C_p\times C_{2p}$. We consider the regular subgroups of $\Hol(N)$ obtained in Section \ref{abelian} and use the notations introduced there. We observe first that by conjugation by $\Aut(N)$, we obtain the $p+1$ conjugates of $\Syl_p$ in $\Hol(N)$. Since $N$ is normal in $\Hol(N)$, it gives one skew brace with $\Soc(B)=\Ann(B)=B$ and $\Aut(B)=\Aut(N)$. For the subgroups $G_i:=\langle (a^i,\varphi),b,c \rangle$, with $0<i\leq p-1$, we obtain $\psi G_1 \psi^{-1}=G_i$, for $\psi$ defined by $\psi(a)=a^i, \psi(b)=b^i$, hence they give one skew brace with $\Soc(B)=\Ann(B)=\langle b,c \rangle \simeq C_{2p}$ and $\Aut(B)=\{ \psi \in \Aut (N) : \psi(a)=a^i b^j, \psi(b)=b^{i^2}, p\nmid i \}$, hence $|\Aut(B)|=p(p-1)$.

We consider now the regular subgroups of $\Hol(N)$ isomorphic to $C_p\times D_{2p}$. The conjugates of $\langle a,b,(c,\chi) \rangle$ are conjugated by $\Aut(N)$, hence they give one skew brace $B$ with $\Soc(B)=\langle a,b \rangle \simeq C_p\times C_p, \Ann(B)=\langle b \rangle \simeq C_p, \Aut(B)=\{ \psi \in \Aut(N): \psi(a)=a^i, \psi(b)=b^l, p\nmid i, p\nmid l \}$, hence $|\Aut(B)|=(p-1)^2$.

For the second family of regular subgroups of $\Hol(N)$ isomorphic to $C_p\times D_{2p}$, we may consider the automorphisms $\psi_1$ defined by $\psi_1(a)=ab, \psi_1(b)=b$ and $\psi_2$ defined by $\psi_2(a)=a^m, \psi_2(b)=m$, for $m$ of order $p-1$ modulo $p$, to obtain that they all belong to the same conjugation class under $\Aut(N)$. Hence they give one skew brace $B$ with $\Soc(B)=\Ann(B)=\langle b \rangle$ and $\Aut(B)=\{ \psi \in \Aut(N) : \psi(a)=a^i, \psi(b)=b^{i^2} \}, 0<i\leq p-1$, and then $|\Aut(B)|=p-1$.

Finally, there is one regular subgroup of $\Hol(N)$ isomorphic to $(C_p\times C_p) \rtimes C_2$, $\langle a,b,(c,\chi) \rangle$ and the corresponding skew brace $B$ satisfies $\Soc(B)=\langle a,b \rangle \simeq C_p\times C_p, \Ann(B)=\{1\}$ and $\Aut(B)=\Aut(N)$.

We summarize the obtained results in the following table.

\begin{center}
\begin{tabular}{|c||c|c|c|c|}
\hline
\multicolumn{5}{|c|}{Braces with additive group isomorphic to $C_p\times C_{2p}$ }\\
\hline \hline
number of braces & $(B,\circ)$  &  $|\Soc(B)|$ & $|\Ann(B)|$ & $|\Aut(B)|$ \\
\hline
\hline
1 & $C_p\times C_{2p}$ & $2p^2$  & $2p^2$ & $p(p+1)(p-1)^2$  \\
\hline
1 & $C_p\times C_{2p}$ & $2p$  & $2p$ & $p(p-1)$  \\
\hline \hline
1 & $C_p\times D_{2p}$ & $p^2$  & $p$ & $(p-1)^2$  \\
\hline
1 & $C_p\times D_{2p}$ & $p$  & $p$ & $p-1$  \\
\hline \hline
1 & $(C_p\times C_p)\rtimes C_2$ & $p^2$  & $1$ & $p(p+1)(p-1)^2$  \\
\hline
\end{tabular}
\end{center}

\subsection{Additive group isomorphic to $C_p\times D_{2p}$}

Let $N$ be the group $C_p\times D_{2p}=\langle r,s,c \rangle$. We consider the regular subgroups of $\Hol(N)$ obtained in Section \ref{cidi} and use the notations introduced there. 

By conjugating by $\varphi$, defined by $\varphi(r)=r, \varphi(s)=rs, \varphi(c)=c$, we obtain that all regular subgroups $A_{p-2,0} \times \langle r^i s \rangle$ are conjugate by $\Aut(N)$ and give one skew brace $B$ with $\Soc(B)=\Ann(B)=\langle c \rangle$ and $\Aut(B)=\{ \psi \in \Aut(N): \psi(r)=r^i, \psi(s)=s, \psi(c)=c^k \}$, hence $|\Aut(B)|=(p-1)^2$. For the second family of regular subgroups of $\Hol(N)$ isomorphic to $C_p\times C_{2p}$, namely $\langle r,c,(s,\chi_{1,i}) \rangle$ we obtain as well that they are all conjugate under $\Aut(N)$ and give one skew brace with $\Soc(B)=\Ann(B)=\langle c \rangle$ and $\Aut(B)=\{ \psi \in \Aut(N): \psi(r)=r^i, \psi(s)=s, \psi(c)=c^k \}$, hence $|\Aut(B)|=(p-1)^2$.

We consider now the regular subgroups of $\Hol(N)$ isomorphic to $C_p\times D_{2p}$. For the family $\langle r,s,(c,\varphi^l)\rangle$, we have that $\langle r,s,c \rangle$ is normal in $\Hol(N)$, whereas the subgroups with $l\neq 0$ are all conjugate under $\Aut(N)$. We obtain then two skew braces. For the first one, we have $\Soc(B)=\Ann(B)=\langle c \rangle$ and $\Aut(B)=\Aut(N)$. For the second one, we have $\Soc(B)=\{1\}$, $\Ann(B)=\{1\}$ and $\Aut(B)=\{ \psi \in \Aut(N): \psi(c)=c^k, \psi(r)=r^k, \psi(s)=r^j s \}, 0<k\leq p-1, 0\leq j \leq p-1$, hence $|\Aut(B)|=p(p-1)$. We consider now the family $\langle (r,\varphi^{-2}), (c,\varphi^l), (s,\chi_{1,0})\rangle$. By conjugating by the automorphism $\psi_1$ defined by $\psi_1(c)=c, \psi_1(r)=r^m, \psi_1(s)=s$, with $m$ of order $p-1$ modulo $p$, we obtain that the subgroups with $l\neq 0$ are conjugated with each other. We obtain then two skew braces. For the one corresponding to $l=0$, we have $\Soc(B)=\Ann(B)=\langle c \rangle, \Aut(B)=\Aut(N)$. For the one corresponding to $l\neq 0$, we have $\Soc(B)=\Ann(B)=\{ 1 \}, \Aut(B)=\{ \psi \in \Aut(N): \psi(c)=c^k, \psi(r)=r^k, \psi(s)=r^js, 0<k\leq p-1, 0\leq j \leq p-1 \}$, hence $|\Aut(B)|=p(p-1)$. We consider now the family $\langle r,(c,\varphi^l),(s,\chi_{2,i})\rangle, 0\leq i,l \leq p-1$. By conjugating by $\varphi$, we obtain that all subgroups with $l=0$ are in the same conjugation class. For $l\neq 0$, by conjugating by $\varphi$ and by $\psi_1$ as above, we obtain a second conjugation class. We obtain then $2$ skew braces. If $B$ is the skew brace corresponding to $l=0$, we have $\Soc(B)=\langle c \rangle, \Ann(B)=\{ 1 \}, \Aut(B)=\Aut(N)$. For $l\neq 0$, we have $\Soc(B)=\Ann(B)=\{ 1 \}, \Aut(B)=\{ \psi \in \Aut(N): \psi(c)=c^k, \psi(r)=r, \psi(s)=r^js \}$, hence $|\Aut(B)|=p(p-1)$. We consider now the family $\langle (cr^{l/2},\Id),(r,\varphi^{-2}),(r^is,\chi_3) \rangle$. By conjugating by $\psi_1$, as above, and by $\psi_2$, defined by $\psi_2(r)=r, \psi_2(c)=c, \psi_2(s)=rs$, we obtain that there are two conjugation classes under $\Aut(N)$. For the skew brace $B$ corresponding to $l=0$, we have $\Soc(B)=\langle c \rangle, \Ann(B)=\{ 1 \}, \Aut(B)=\{ \psi \in \Aut(N): \psi(c)=c^k, \psi(r)=r^m, \psi(s)=s, 1\leq k,m \leq p-1 \}$, hence $|\Aut(B)|=(p-1)^2$. For the skew brace $B$ corresponding to $l\neq 0$, we have $\Soc(B)=\langle c \rangle, \Ann(B)=\{ 1 \}, \Aut(B)=\{ \psi \in \Aut(N): \psi(c)=c^k, \psi(r)=r^k, \psi(s)=r^ms, 1\leq k\leq p-1, 0\leq m \leq p-1  \}$, hence $|\Aut(B)|=p(p-1)$.

We consider now the regular subgroups of $\Hol(N)$ isomorphic to $(C_p\times C_p) \rtimes C_2$. For $G=\langle (r,\varphi^{-2}),c,(s,\chi_{2,0} \rangle$, we obtain one skew brace $B$ with $\Soc(B)=\langle c \rangle, \Ann(B)=\{ 1 \}, \Aut(B)=\Aut(N)$. For $G=\langle r,c, (s,\chi_3) \rangle$, we obtain one skew brace $B$ with $\Soc(B)=\langle c \rangle, \Ann(B)=\{ 1 \}, \Aut(B)=\Aut(N)$.

We summarize the obtained results in the following table.

\begin{center}
\begin{tabular}{|c||c|c|c|c|}
\hline
\hline
\multicolumn{5}{|c|}{Skew braces with additive group isomorphic to $C_p\times D_{2p}$ }\\
\hline \hline
number of braces & $(B,\circ)$  &  $|\Soc(B)|$ & $|\Ann(B)|$ & $|\Aut(B)|$ \\
\hline
\hline
2 & $C_p\times C_{2p}$ & $p$  & $p$ & $(p-1)^2$  \\
\hline \hline
2 & $C_p\times D_{2p}$ & $p$  & $p$ & $p(p-1)^2$  \\
\hline
3 & $C_p\times D_{2p}$ & $1$  & $1$ & $p(p-1)$  \\
\hline
1 & $C_p\times D_{2p}$ & $p$  & $1$ & $p(p-1)^2$  \\
\hline
1 & $C_p\times D_{2p}$ & $p$  & $1$ & $(p-1)^2$  \\
\hline
1 & $C_p\times D_{2p}$ & $p$  & $1$ & $p(p-1)$  \\
\hline \hline
2 & $(C_p\times C_p)\rtimes C_2$ & $p$  & $1$ & $p(p-1)^2$  \\
\hline
\end{tabular}
\end{center}

\subsection{Additive group isomorphic to $(C_p\times C_{p})\rtimes C_2$}
Let $N$ be the group $(C_p\times C_{p})\rtimes C_2$. For any skew brace $B$ with additive group isomorphic to $N$, we have $\Soc(B)=\Ann(B)=\{ 1 \}$, by Proposition \ref{soc}, since $Z(N)$ is trivial.  We consider the regular subgroups of $\Hol(N)$ obtained in Section \ref{semi} and use the notations introduced there.

\subsubsection{Multiplicative group isomorphic to $C_p\times C_{2p}$}

We consider the regular subgroups of $\Hol(N)$ isomorphic to $C_p\times C_{2p}$. We examine first those having an element of order 2 of the form $(z,\Id)$. By conjugating with $\psi_1$ defined by $\psi_1(a)=a, \psi_1(b)=b,\psi_1(c)=a^mb^nc$, we obtain that all subgroups $\langle (a,\varphi_2^{-2}),(b,\varphi_3^{-2}),(a^mb^nc,\Id) \rangle$ are conjugate. For the corresponding skew brace $B$, we obtain $\Aut(B)=\{ \psi \in \Aut(N): \psi(a)=a^ib^j, \psi(b)=a^kb^l \}$, hence $|\Aut(B)|=(p^2-1)(p^2-p)$. We consider the family $\langle (a,\varphi_2^{-2}),(b,\varphi_1^{j_1}\varphi_2^{(2-n)j_1}\varphi_3^{-2}),(a^mb^nc,\Id)\rangle$. By conjugating with $\psi_1$ and $\psi_2$ defined by $\psi_2(a)=a, \psi_2(b)=b^l,\psi_2(c)=c$, with $l$ of order $p-1$ modulo $p$, we obtain that all subgroups in this family are conjugate. For the corresponding skew brace $B$, we have $\Aut(B)=\{ \psi \in \Aut(N): \psi(a)=a^{l^2}, \psi(b)=a^kb^l,\psi(c)=c \}$, hence $|\Aut(B)|=p(p-1)$.

We examine next the regular subgroups with an element of order 2 of the form $(z,\chi_1)$, with $\chi_1$ defined in (\ref{chi1}). We consider the family $\langle (a,\Id),(b,\varphi_2^{k}\varphi_3^{-2}),(a^mb^nc,\chi_1)\rangle$, with $\chi_1$ defined by $\chi_1(a)=a^{-1}, \chi_1(b)=a^{-k}b, \chi_1(c)=a^{2m+kn}c$. By conjugating with $\psi_1$ and $\varphi_1$, we obtain that all subgroups in this family are conjugate. For the corresponding skew brace $B$, we have $\Aut(B)=\{ \psi \in \Aut(N): \psi(a)=a^{i}, \psi(b)=b^l,\psi(c)=c \}$, hence $|\Aut(B)|=(p-1)^2$. Now the family $\langle (a,\varphi_2^{-2}),(b,\varphi_2^{k}),(a^mb^nc,\chi_1)\rangle$, with $\chi_1$ defined by $\chi_1(a)=a, \chi_1(b)=a^{-k}b^{-1}, \chi_1(c)=a^{kn}b^{2n}c$, is the conjugate of the preceding one by $\psi$ defined by $\psi(a)=a^kb^2, \psi(b)=a^{1+(k^2/2)}b^k, \psi(c)=c$. We consider the family $\langle (a,\varphi_2^i\varphi_3^j),(b,\varphi_2^k\varphi_3^l),(a^mb^nc,\chi_1) \rangle$. By conjugating with $\psi_1, \varphi_1$ and the automorphism $\psi$ defined by $\psi(a)=a, \psi(b)=b^{\alpha}, \psi(c)=c$, we obtain that all subgroups in this family are conjugate. Moreover, with $\psi$ defined by $\psi(a)=a^2b,\psi(b)=b,\psi(c)=c$, we have $\psi \langle (a,\Id),(b,\varphi_3^{-2}),(c\chi_1)\rangle \psi^{-1}=\langle (a,\varphi_3),(b,\varphi_3^{-2}),(c\chi_1')\rangle$, where $\chi_1(a)=a^{-1}, \chi_1(b)=b,\chi_1(c)=c$ and $\chi_1'(a)=a^{-1}b^{-1}, \chi_1'(b)=b,\chi_1'(c)=c$. Hence the groups in this family are conjugated to the subgroups in the preceding ones.

We examine next the regular subgroups with an element of order 2 of the form $(z,\chi_2)$, with $\chi_2$ defined in (\ref{chi2}). We consider the family $\langle (a,\Id),(b,\Id),(a^{u/2}b^{v/2}c,\chi_2)\rangle$. By conjugating with $\psi_3$ defined by $\psi_3(a)=a,\psi_3(b)=b,\psi_3(c)=a^{u/2}b^{v/2}c$, we obtain that this family is contained in one conjugation class. For the corresponding skew brace $B$, we obtain $\Aut(B)=\{ \psi \in \Aut(N): \psi(a)=a^i b^j, \psi(b)=a^kb^l, \psi(c)=c \}$, hence $\Aut(B)=(p^2-1)(p^2-p)$. Finally for the family $\langle (a,\Id),(b,\varphi_1^{j_1} \varphi_2^{-vj_1/2}),(a^{u/2}b^{v/2}c,\chi_2)\rangle$, by conjugating with $\psi_3$ and by $\psi$ defined by $\psi(a)=a^{\alpha},\psi(b)=b,\psi(c)=c$, with $\alpha$ of order $p-1$ modulo $p$, we obtain that this family is contained in one conjugation class. For the corresponding skew brace $B$, we have $\Aut(B)=\{ \psi \in \Aut(N): \psi(a)=a^{l^2}, \psi(b)=a^kb^l, \psi(c)=c \}$, hence $\Aut(B)=p(p-1)$.

\subsubsection{Multiplicative group isomorphic to $C_p\times D_{2p}$}

We consider the regular subgroups of $\Hol(N)$ isomorphic to $C_p\times D_{2p}$. We examine first those having an element of order 2 of the form $(z,\Id)$. By conjugating with $\varphi_1, \varphi_2, \varphi_3$ and $\psi$ defined by $\psi(a)=ab, \psi(b)=b, \psi(c)=c$, we obtain that all subgroups in the family $\langle (a,\varphi_2^{i}\varphi_3^j),(b,\varphi_2^k\varphi_3^l),(a^mb^nc,\Id) \rangle$, with $i+l=-2, jk=il$ are in the same conjugation class. For the corresponding skew brace, with $j=k=m=n=0, i=-2$, we obtain $\Aut(B)=\{ \psi \in \Aut(N): \psi(a)=a^i, \psi(b)=b^l, \psi(c)=b^n c\}$, hence $|\Aut(B)=p(p-1)^2$. We consider now the family $\langle (a,\Id), (b,\varphi_1^{j_1}\varphi_2^{j_2}\varphi_3^{-2}),(b^nc,\Id)\rangle$. By conjugating with $\varphi_3, \varphi_1$ and $\psi$ defined by $\psi(a)=a^{\alpha}, \psi(b)=b, \psi(c)=c$, with $\alpha$ of order $p-1$ modulo $p$, we obtain that all subgroups in this family are conjugated with each other. For the corresponding skew brace $B$, with $j_1=1, j_2=n=0$, we obtain $\Aut(B)=\{ \psi \in \Aut(N): \psi(a)=a^{l^2}, \psi(b)=b^l, \psi(c)=a^mc \}$, hence $|\Aut(B)|=p(p-1)$.

We examine next the regular subgroups with an element of order 2 of the form $(z,\chi_1)$, with $\chi_1$ defined in (\ref{chi1}). We consider the family $\langle (a,\Id),(b,\varphi_2^k),(b^nc,\chi_1) \rangle$, with $\chi_1$ defined by $\chi_1(a)=a, \chi_1(b)=a^{-k} b^{-1}, \chi_1(c)=a^{kn}b^{2n}c$. Conjugating with $\varphi_3$ and $\psi$ as above, we obtain that the subgroups in this family are grouped in two conjugation classes depending on whether $k=0$ or $k\neq 0$. For the skew brace $B$ corresponding to $k=n=0$, we obtain $\Aut(B)= \{ \psi \in \Aut(N): \psi(a)=a^i, \psi(b)=b^l, \psi(c)=a^m c \}$, hence $|\Aut(B)|=p(p-1)^2$. For the skew brace $B$ corresponding to $k=1, n=0$, we obtain $\Aut(B)= \{ \psi \in \Aut(N): \psi(a)=a^i, \psi(b)=b^i, \psi(c)=a^m c \}$, hence $|\Aut(B)|=p(p-1)$. We consider now the family $\langle (a,\varphi_2^{-2}),(b,\varphi_2^k\varphi_3^{-2}),(a^mb^nc,\chi_1) \rangle$, with $\chi_1$ defined by $\chi_1(a)=a^{-1}, \chi_1(b)=a^{-k} b, \chi_1(c)=a^{2m+kn}c$. Conjugating with $\varphi_2, \varphi_3$ and $\psi$ as above, we obtain that the subgroups in this family are grouped in two conjugation classes depending on whether $k=0$ or $k\neq 0$. For the skew brace $B$ corresponding to $k=m=n=0$, we obtain $\Aut(B)= \{ \psi \in \Aut(N): \psi(a)=a^i, \psi(b)=b^l, \psi(c)=a^m c \}$, hence $|\Aut(B)|=p(p-1)^2$. For the skew brace $B$ corresponding to $k=1, m=n=0$, we obtain $\Aut(B)= \{ \psi \in \Aut(N): \psi(a)=a^i, \psi(b)=b^l, \psi(c)=a^m c \}$, hence $|\Aut(B)|=p(p-1)$. We consider the family $\langle (a,\varphi_2^i \varphi_3^j),(b,\varphi_2^k\varphi_3^i),(a^mb^nc,\chi_1) \rangle$, with $j\neq 0, jk=i(i+2)$ and $\chi_1$ defined by $\chi_1(a)=a^{i+1} b^j, \chi_1(b)=a^{-k} b^{-i-1}, \chi_1(c)=a^u b^v c$, with $ju=iv, jm=(i+2)n-v$. By conjugation with  $\varphi_1, \varphi_2, \varphi_3$, we obtain that the subgroups with the same value of $i$ are conjugated. For the corresponding skew brace $B$ with $j=1, m=n=0$, we obtain $\Aut(B)=\{ \psi \in \Aut(N): \psi(a)=a^{\alpha}, \psi(b)=b^{\alpha}, \psi(c)=a^{\gamma} c \}$, hence $|\Aut(B)|=p(p-1)$. We may check that skew braces corresponding to different values of $i$ are non-isomorphic. We consider the family $\langle (a,\varphi_2^{-2}),(b,\varphi_1^{j_1}\varphi_2^{j_2}\varphi_3^{-2}),(a^mb^nc,\chi_1)\rangle$, with $j_1=1$ or $n=2$ and with $\chi_1$ defined by $\chi_1(a)=a, \chi_1(b)=a^t b^{-1}, \chi_1(c)=a^{-tn} b^{2n} c$. By conjugating with $\varphi_1, \varphi_2, \varphi_3$ and $\psi$ defined by $\psi(a)=a^{-1}, \psi(b)=b, \psi(c)=c$, we obtain that all subgroups in this family lie in one conjugation class. We compute the automorphism group for $j_1=1, j_2=m=n=t=0$ and obtain $\Aut(B)=\{ \psi \in \Aut(N): \psi(a)=a^{l^2}, \psi(b)=b^l, \psi(c)=c \}$, hence $|\Aut(B)|=p-1$. We consider the family $\langle (a,\Id),(b,\varphi_1^{j_1}\varphi_2^{j_2}),(c,\chi_1)\rangle$, with $j_1 \neq 0$, $\chi_1$ defined by $\chi_1(a)=a^{-1}, \chi_1(b)=a^t b, \chi_1(c)=a^u c$. By conjugating with $\varphi_1, \varphi_2, \varphi_3$ and $\psi$ defined by $\psi(a)=a^{\alpha}, \psi(b)=b, \psi(c)=c$, we obtain that all subgroups in this family lie in one conjugation class. We consider the corresponding skew brace $B$ for $j_1=1, j_2=t=u=0$. We obtain $\Aut(B)=\{ \psi \in \Aut(N): \psi(a)=a^i, \psi(b)=b^l, \psi(c)=b^nc \}$, hence $|\Aut(B)|=p(p-1)^2$.

We examine next the regular subgroups with an element of order 2 of the form $(z,\chi_2)$, with $\chi_2$ defined in (\ref{chi2}). By conjugating with $\varphi_3$, we obtain that the subgroups in the family $\langle (a,\varphi_2^{-2}),(b,\Id),(b^{v/2}c,\chi_2)\rangle$, with $\chi_2$ defined by $\chi_2(a)=a^{-1}, \chi_2(b)=b^{-1}, \chi_2(c)=b^v c$ are all conjugated with each other. For $v=0$, we obtain $\Aut(B)=\{ \psi \in \Aut(N): \psi(a)=a^i, \psi(b)=b^l, \psi(c)=a^m c\}$, hence $|\Aut(B)|=p(p-1)^2$. Finally we consider the family $\langle (a,\varphi_2^{-2}),(b,\varphi_1^{j_1}\varphi_2^{j_2}),(b^{v/2}c,\chi_2)\rangle$, with $j_1 \neq 0$, $\chi_2$ as in the preceding case. By conjugating with $\varphi_1, \varphi_3$ and $\psi$ defined by $\psi(a)=a^{\alpha}, \psi(b)=b, \psi(c)=c$, we obtain that all subgroups in this family lie in one conjugation class. For $v=j_2=0, j_1=1$, we obtain $\Aut(B)=\{ \psi \in \Aut(N): \psi(a)=a^i, \psi(b)=b^l, \psi(c)=a^m c\}$, hence $|\Aut(B)|=p(p-1)^2$.

\subsubsection{Multiplicative group isomorphic to $(C_p\times C_{p})\rtimes C_2$}

We consider the regular subgroups of $\Hol(N)$ isomorphic to $N$. The subgroup $N$ provides a skew brace $B$ with multiplicative group equal to the additive group, hence $\Aut(B)=\Aut(N)$.

We examine next the regular subgroups with an element of order 2 of the form $(z,\chi_1)$, with $\chi_1$ defined in (\ref{chi1}). We consider the family $\langle (a,\varphi_2^i\varphi_3^j),(b,\varphi_2^k\varphi_3^{-i-2}),(c,\chi_1)\rangle$, with $jk=-i(i+2)$ and $\chi_1$ defined by $\chi_1(a)=a^{i+1}b^j, \chi_1(b)=a^kb^{-i-1}, \chi_1(c)=c$. By conjugation with $\varphi_1, \psi$ defined by $\psi(a)=ab, \psi(b)=b, \psi(c)=c$ and $\psi_{\beta}$ defined by $\psi(a)=a, \psi(b)=b^{\beta}, \psi(c)=c$, we obtain that all subgroups in this family lie in one conjugation class. We consider the corresponding skew  brace $B$ for $i=j=k=0$ and obtain $\Aut(B)= \{ \psi \in \Aut(N): \psi(a)=a^{\alpha} b^{\beta}, \psi(b)=b^{\delta}, \psi(c)=a^{\alpha} c \}$, hence $|\Aut(B)|=p(p-1)^2$. We consider the family $\langle (a,\Id),(b,\varphi_1^{j_1}\varphi_2^{j_2}\varphi_3^{-2}),(b^nc,\chi_1)\rangle$, with $j_1\neq 0$, $j_1n=j_1-j_2$ and $\chi_1$ defined by $\chi_1(a)=a, \chi_1(b)=b^{-1}, \chi_1(c)=b^{2n}c$. By conjugation with $\varphi_3$ and $\psi_{\alpha}$ defined by $\psi(a)=a^{\alpha}, \psi(b)=b, \psi(c)=c$, we obtain that all subgroups in this family lie in one conjugation class. We consider the corresponding skew brace $B$ for $j_1=j_2=1, n=0$ and obtain $\Aut(B)= \{ \psi \in \Aut(N):\psi(a)=a^{\delta^2} , \psi(b)=a^{-\delta \nu/2} b^{\delta}, \psi(c)=a^{\mu} b^{\nu}c \}$, hence $|\Aut(B)|=p^2(p-1)$. We consider the family $\langle (a,\varphi_2^{-2}),(b,\varphi_1^{j_1}\varphi_2^{j_2}),(c,\chi_1)\rangle$, with $j_1\neq 0$ and $\chi_1$ defined by $\chi_1(a)=a^{-1}, \chi_1(b)=a^{j_2}b, \chi_1(c)=c$. By conjugation with $\psi_{\alpha}$ and $\varphi_1$ we obtain that all subgroups in this family lie in one conjugation class. For the corresponding skew brace $B$, with $j_1$=1, $j_2=0$, we obtain $\Aut(B)=\{ \psi \in \Aut(N): \psi(a)=a, \psi(b)=a^{\gamma} b, \psi(c)=a^{\mu} b^{2\gamma} c \}$, hence $|\Aut(B)|=p^2$.

Finally for the skew brace $B$ corresponding to the subgroup $\langle (a,\varphi_2^{-2}),(b,\varphi_3^{-2}),(c,\chi_2)$, with $\chi_2$ defined by $\chi_2(a)=a^{-1}, \chi_2(b)=b^{-1}, \chi_2(c)=c$, we obtain $\Aut(B)= \Aut(N)$.

\vspace{0.5cm}
We summarize the obtained results in the following table.

\begin{center}
\begin{tabular}{|c||c|c|}
\hline
\multicolumn{3}{|c|}{Skew braces with additive group isomorphic to $(C_p\times C_{p})\rtimes C_2$ }\\
\hline \hline
number of braces & $(B,\circ)$  &  $|\Aut(B)|$ \\
\hline
\hline
2 & $C_p\times C_{2p}$ & $p(p+1)(p-1)^2$  \\
\hline
2 & $C_p\times C_{2p}$ &  $p(p-1)$  \\
\hline
1 & $C_p\times C_{2p}$ & $(p-1)^2$  \\
\hline \hline
6 & $C_p\times D_{2p}$  & $p(p-1)^2$  \\
\hline
$p+3$ & $C_p\times D_{2p}$ & $p(p-1)$  \\
\hline
1 & $C_p\times D_{2p}$ &  $p-1$  \\
\hline \hline
2 & $(C_p\times C_p)\rtimes C_2$ &  $p^3(p+1)(p-1)^2$  \\
\hline
1 & $(C_p\times C_p)\rtimes C_2$ & $p(p-1)^2$  \\
\hline
1 & $(C_p\times C_p)\rtimes C_2$ & $p^2(p-1)$  \\
\hline
1 & $(C_p\times C_p)\rtimes C_2$ & $p^2$  \\
\hline
\end{tabular}
\end{center}

\subsection{Summary}
Let $p$ be an odd prime number. As a summary of the preceding results, we give in the following table the number of isomorphism classes of skew braces $(B,\cdot,\circ)$ of order $2p^2$ with additive group $(B,\cdot)$ and multiplicative group $(B,\circ)$.

\begin{center}
\begin{tabular}{|c||c|c|c|c|c|}
\hline
\multicolumn{6}{|c|}{Number of skew braces $(B,\cdot,\circ)$ }\\
\hline \hline
$(B,\circ)$  $\diagdown$ $(B,\cdot)$ &  $C_{2p^2}$& $D_{2p^2}$ & $C_p\times C_{2p}$ & $C_p \times D_{2p}$ & $(C_p\times C_p) \rtimes C_2$ \\
\hline
\hline
$C_{2p^2}$ & $2$  & $4$ & $0$ & $0$ & $0$ \\
\hline
$D_{2p^2}$ &   $1$  & $2$ & $0$ & $0$ & $0$ \\
\hline
$C_p\times C_{2p}$ & $0$ & $0$ & $2$ & $2$ & $5$ \\
\hline
$C_p\times D_{2p}$ & $0$ & $0$ & $2$ & $8$ & $p+10$ \\
\hline
$(C_p\times C_p)\rtimes C_2$ & $0$ & $0$ & $1$ & $2$  & $5$ \\

\hline
\end{tabular}
\end{center}

\section*{Acknowledgements } I would like to thank Ilaria Del Corso for pointing to me a mistake in a previous version of this paper and the referee for his/her valuable comments.

This work was supported by grant PID2019-107297GB-I00 (MICINN).

\end{document}